\documentclass{elsarticle}

\usepackage{amscd}
\usepackage{amsmath}
\usepackage{amssymb}
\usepackage{amsthm}
\usepackage{graphicx}
\usepackage[all]{xy}
\usepackage{kotex}
\usepackage{color}
\usepackage{enumerate}
\theoremstyle{plain}

\usepackage{nameref}

\newtheorem{Theorem}{Theorem}[section]
\newtheorem{Lemma}[Theorem]{Lemma}

\newtheorem{definition}[Theorem]{Definition}
\newtheorem*{Theorem*}{Theorem}
\theoremstyle{remark}
\newtheorem{Remark}[Theorem]{\bf{Remark}}
\newtheorem{example}[Theorem]{\bf{Example}}

\usepackage{color}							

\journal{....}
\begin{document}
	\begin{frontmatter}
		\title{A complete characterization of the blow-up solutions to discrete $p$-Laplacian parabolic equations with $q$-reaction under the mixed boundary conditions\\}

		\author{Jaeho Hwang}
		\ead{hjaeho@sogang.ac.kr}
		
		\cortext[cjhp]{Corresponding author}
		\address{Department of Mathematics, Sogang University, Seoul 04107, Republic of Korea}

		\begin{abstract}
			In this paper, we consider discrete $p$-Laplacian parabolic equations with $q$-reaction term under the mixed boundary condition and the initial condition as follows:
			\begin{equation*}
			\begin{cases}
			u_{t}\left(x,t\right) = \Delta_{p,\omega} u\left(x,t\right) +\lambda  \left\vert u\left(x,t\right) \right\vert^{q-1} u\left(x,t\right), &\left(x,t\right) \in S \times \left(0,\infty\right), \\ \mu(z)\frac{\partial u}{\partial_{p} n}(z)+\sigma(z)\vert u(z)\vert^{p-2}u(z)=0, &\left(x,t\right) \in \partial S \times \left[0,\infty\right), \\ u\left(x,0\right) = u_{0}(x) \geq 0, &x \in \overline{S}.
			\end{cases}
			\end{equation*}
			 where $p>1$, $q>0$, $\lambda>0$ and $\mu,\sigma$ are nonnegative functions on the boundary $\partial S$ of a network $S$, with $\mu(z)+\sigma(z)>0$, $z\in\partial S$. Here, $\Delta_{p,\omega}$ and $\frac{\partial \phi}{\partial_{p} n}$ denote the discrete $p$-Laplace operator and the $p$-normal derivative, respectively. The parameters $p>1$ and $q>0$ are completely characterized to see when the solution blows up, vanishes, or exists globally. Indeed, the blow-up rates when blow-up does occur are derived. Also, we give some numerical illustrations which explain the main results.
		\end{abstract}
		
	\end{frontmatter}

\section{Introduction}\label{intro}
The $p$-Laplacian parabolic equations with $q$-reaction term (reaction-diffusion equations)
\begin{equation*}
	u_{t}=\mathrm{div}(|\nabla u|^{p-2}\nabla u)+\lambda|u|^{q-1}u,
\end{equation*}
have been investigated by a lot of researchers. They study above equations under the Dirichlet boundary condition, the Neumann boundary condition, the Robin boundary condition, and so on, which have found many applications in chemical reactions, electronic models, and biological phenomena (see \cite{YJ, LX, C2, L}).

In recent years, there are researchers who study the reaction-diffusion equations under the boundary conditions which mix Dirichlet boundary condition and Neumann boundary condition or Dirichlet boundary condition and Robin boundary condition (see \cite{K,GMPR,Z}). From this motivation, we consider the mixed boundary conditions which include represent boundary conditions for each boundary point.

In this paper, we discuss the discrete $p$-Laplacian parabolic equations with $q$-reaction term under the mixed boundary conditions as follows:
\begin{equation}\label{equation}
	\begin{cases}
	u_{t}\left(x,t\right) = \Delta_{p,\omega} u\left(x,t\right) +\lambda  \left\vert u\left(x,t\right) \right\vert^{q-1} u\left(x,t\right), &\left(x,t\right) \in S \times \left(0,\infty\right),\\
	B[u]=0, &\text{on}\,\,\partial S \times \left[0,\infty\right),\\
	u\left(x,0\right) = u_{0}(x) \geq 0, &x \in \overline{S}
	\end{cases}
\end{equation}
where $p>1$, $q>0$, $\lambda>0$, and $B[u]=0$ on $\partial S\times (0,\infty)$ stands for the boundary condition
\begin{equation}\label{boundary}
	\mu(z)\frac{\partial u}{\partial_{p} n}(z)+\sigma(z)\vert u(z)\vert^{p-2}u(z)=0.
\end{equation}
Here, $\mu$ and $\sigma$ are nonnegative functions on the boundary $\partial S$ of a network $S$, with $\mu(z)+\sigma(z)>0$ for all $z\in \partial S$. Here, $\Delta_{p,\omega}$ and $\frac{\partial \phi}{\partial_{p} n}$ denote the discrete $p$-Laplace operator and the $p$-normal derivative, respectively (which will be introduced in Section 1). It is easy to see that the boundary condition \eqref{boundary} includes the various boundary conditions such as the Dirichlet boundary condition, the Neumann boundary condition, the Robin boundary condition, and so on. We note here that one of the meaning of our result is an unified approach.

As far as the authors know, it seems that there have been no paper which deal with the $p$-Laplacian parabolic equations under the above mixed boundary conditions, in the discrete case, not even in the continuous case. Therefore, it is expected that our methods will be obtained more interesting results in the discrete and continuous case.

The aim of this paper is to characterize `completely' the parameters $p>1$, $q>0$ $\mu$, $\sigma$, and $\lambda>0$ to see when the solutions to the equation \eqref{equation} blows up, vanishes, or exists globally.

In conclusion, main result of this paper is divided into two cases.\\
Case 1: $\sigma\equiv 0$ (Neumann boundary condition).\\
In this case, the solution to the equation \eqref{equation} blows up in finite time $T$ if and only if  $q>1$, for every $\lambda>0$ and nontrivial nonnegative initial data $u_{0}$.\\
Case 2: $\sigma\not\equiv 0$.\\
In the case of $\sigma\not\equiv 0$, we summarize the result as following:
\newpage
\begin{figure}[ht]
	\centering
	\includegraphics[width=10.66cm,height=6cm]{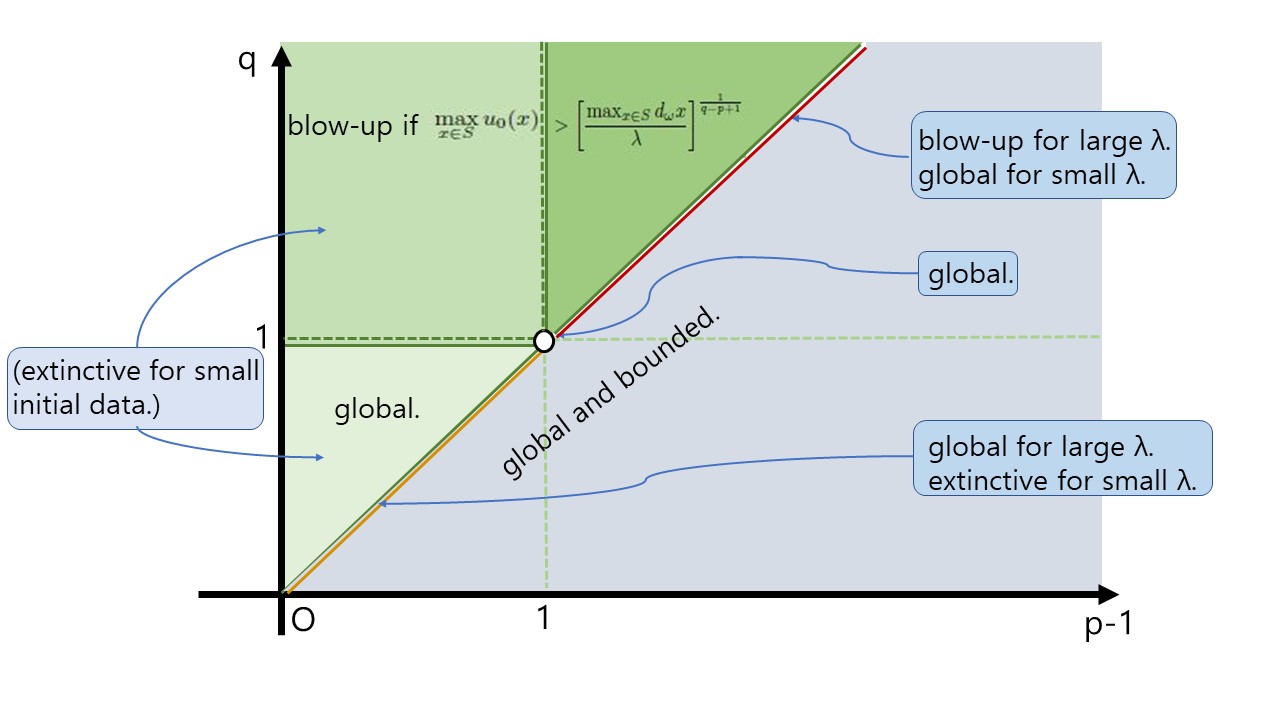}
	\\Figure 0. A complete characterization of $p$ and $q$.
\end{figure}

As seen in the Figure 0, we obtain the blow-up solutions for $0<p-1<q$ and $q>1$ whenever the initial data $u_{0}$ is sufficiently large that
\begin{equation*}
	\max_{x\in S}u_{0}(x)>\left[\frac{\max_{x\in S}d_{\omega}x}{ \lambda}\right]^\frac{1}{q-p+1}.
\end{equation*}
Here, $d_{\omega}x:=\sum_{y\in \overline{S}}\omega(x,y)$ (which will be introduced in Section 1). Also, in the case $p-1=q>0$, we obtain the exact condition that when the solution blows up, exists globally, and vanishes. As a matter of fact, there have been no paper which deal with the blow-up or extinctive solutions to the equation \eqref{equation} completely in the continuous version.

Even though we discussed here the equation \eqref{equation} only in the discrete settings, instead of the continuous settings, we believe that our results are not only interesting in itself, but also may help to study the equation \eqref{equation} in the continuous settings, since the continuous version is basically approximated by the discrete version by way of numerical schemes.

We organized this paper as follows. In section 1, we discuss the preliminary concepts on networks and local existence of the solution to the equation \eqref{equation}. In section 2, we investigate discrete version of comparison principles. In section 3, we are devoted to find out blow-up condition and extinctive condition of the solution. Also, we have blow-up set and blow-up rate with the blow-up time. Finally, in section 4, we give some numerical experiments to explain our main results.

\section{Preliminaries and Discrete Comparison Principles}\label{preli}
In this section, we start with the theoretic graph notions frequently used throughout this paper (see \cite{CB1,CB2}, for more details).
\begin{definition}
	\begin{enumerate}[(i)]
		\item  A graph $G=G\left(V,E\right)$ is a finite set $V$ of $vertices$(or $nodes$) with a set $E$ of $edges$ (two-element subsets of $V$). We simply denote by $|G|$ the number of vertices in $G$. Conventionally, we denote by $x\in V$ or $x\in G$ the fact that $x$ is a vertex in $G$. Moreover, by $\{x,y\}\in E$ we mean that an edge with endpoints $x$ and $y$ and by $x\sim y$ we mean that $x$ and $y$ are connected by an edge, i.e. $x$ and $y$ are adjacent.
		\item A graph $G$ is called $simple$ if it has neither multiple edges nor loops.
		\item A graph $G$ is called $connected$ if for every pair of vertices $x$ and $y$, there exists a sequence(called a $path$) of vertices $x=x_{0},x_{1},\cdots,x_{n-1},x_{n}=y$ such that $x_{j-1}$ and $x_{j}$ are connected by an edge for $j=1,\cdots,n$.
		\item A graph $G'=G'\left(V',E'\right)$ is called a $subgraph$ of $G\left(V,E\right)$ if $V'\subset V$ and $E'\subset E$. In this case, $G$ is called a host graph of $G'$. If $E'$ consists of all the edges from $E$ which connect the vertices of $V'$ in its host graph $G$, then $G'$ is called an induced subgraph.
	\end{enumerate}
\end{definition}
Throughout this paper, by a graph $G(V, E)$ we mean that it is a connected and simple. 
\begin{definition}
	A $weight$ on a graph $G$ is a symmetric function $\omega\,:\,V\times V\rightarrow\left[0,+\infty\right)$ satisfying the following:
	\begin{enumerate}[(i)]
		\item $\omega\left(x,x\right)=0$, \,\,$x\in V$,
		\item $\omega\left(x,y\right)=\omega\left(y,x\right)$,\,\, $x$, $y\in V$ 
		\item $\omega\left(x,y\right)>0$ if and only if $x\sim y$,
	\end{enumerate}
	and a graph $G$ with a weight $\omega$ is called a weighted graph or a $network$.
\end{definition}

\begin{definition}
	Let $S(V',E')$ be an induced subgraph of a graph $G(V,E)$. By  $\partial S:=\partial S(\partial V',\partial E')$, so called a boundary of $S$, we mean a subgraph whose vertices and edges are given by
	\[
	\begin{aligned}
	&\partial V':=\{z\in V\setminus V' \,|\, y\sim z\text{ for some }y\in V' \},\\
	&\partial E':=\{ \{x,y\}\in E\setminus E' \,|\, x\text{ or }y \in V' \},
	\end{aligned}
	\]
	respectively.
\end{definition}
Let $S$ is a connected induced subgraph of a graph $G(V,E)$. By a network $\overline{S}$(or $S\cup\partial S$) we mean that it is a subgraph of a graph $G(V,E)$ with a weight $\omega$  whose vertices and edges are consisting of all those in  $S$ or $\partial S$.
\begin{definition}
	The degree $d_{\omega}x$ at a vertex $x$ in $\overline{S}$ is defined by
	\[
	d_{\omega}x:=\sum_{y\in\overline{S}}\omega\left(x,y\right).
	\]
\end{definition}


We now introduce notations for a calculus on graphs. From now on, by a function on $T$, we mean that it is a real valued function defined on the vertices of the graph $T$. 

\begin{definition}
	Let $p>1$. Suppose that $u$ is a function on  $\overline{S}$.
	\begin{enumerate}[(i)]
		\item The $p$-directional derivative of a function $u$ at a vertex $x$ in the direction of $y$ is defined by 
		\[
		D_{y}u(x):=\left|u(z)-u(y)\right|^{p-2}[u(y)-u(x)]\sqrt{\omega(x,y)}.
		\]   
		\item The $p$-gradient $\nabla_{p,\omega}$ of a function $u$ at a vertex $x\in \overline{S}$ is defined by 
		\[
		\nabla_{p,\omega} u(x):=\left(D_{y}u(x)\right)_{y\in \overline{S}}.
		\]
		\item The (outward) $p$-normal derivative of a function $u$ at $z$ in $\partial S$ is defined by 
		\[
		\frac{\partial u}{\partial_{p} n}(z):=\sum_{y\in S}\left|u(z)-u(y)\right|^{p-2}\left[u(z)-u(y)\right]\omega(z,y).
		\]
		\item The discrete $p$-Laplacian $\Delta_{p,\omega}$ of a function $u$ at a vertex $x\in \overline{S}$ is defined by
		\[
		\Delta_{p,\omega}u\left(x\right):=\sum_{y\in \overline{S}}\left|u(z)-u(y)\right|^{p-2}\left[u\left(y\right)-u\left(x\right)\right]\omega\left(x,y\right).
		\]
	\end{enumerate}
\end{definition}

The following theorem is useful throughout this paper.

\begin{Theorem}[See \cite{KC}]\label{laplace theorem}
	Let $p>1$. For functions $f$ and $g$ on $\overline{S}$, we have
\begin{equation*}
\begin{aligned}
(i)\,& 2\sum_{x\in\overline{S}}g(x)[-\Delta_{p,\omega}f(x)]\\& = \sum_{x,y\in\overline{S}} \left\vert f(y)-f(x)\right\vert^{p-2} \left[f(y)-f(x)\right]\left[g(y)-g(x)\right]\omega(x,y).\\ (ii)\,&2\sum_{x\in\overline{S}}f(x)[-\Delta_{p,\omega}f(x)] = \sum_{x,y\in\overline{S}} \left\vert f(y)-f(x)\right\vert^{p}\omega(x,y).
\end{aligned}
\end{equation*}
\end{Theorem}

\begin{Lemma}[See \cite{CH}]\label{eigenvalue}
	
	For $p>1$, there exist $\lambda_{p,0}>0$ and a function $\phi_{0}\left(x\right)>0$, $x\in S\cup\Gamma$
	such that
	\[
	\begin{cases}
	-\Delta_{p,\omega}\phi_{0}\left(x\right)=\lambda_{p,0}|\phi_{0}(x)|^{p-2}\phi_{0}\left(x\right), & x\in S,\\
	B[\phi_{0}]=0, & \text{on} \hspace{2mm}\partial S,
	\end{cases}
	\]
	where $B[\phi_{0}]=0$ on $\partial S$ stands for the boundary condition
	\begin{equation*}
	\mu(z)\frac{\partial \phi_{0} }{\partial_{p} n}(z)+\sigma(z)|\phi_{0}(z)|^{p-2}\phi_{0}(z)=0,\,\,z\in \partial S.
	\end{equation*}
	Here, $\Gamma:=\{z\in \partial S\,|\,\mu(z)>0 \}$ (which will be used throughout this paper) and $\mu,\sigma:\partial S\rightarrow[0,+\infty)$ are functions with $\mu(z)+\sigma(z)>0$ for all $z\in \partial S$. Moreover, $\lambda_{p,0}$ is given by
	\[
	\begin{aligned}
	\lambda_{p,0} =& \min_{u\in\mathcal{A},u\not\equiv0}\frac{\displaystyle\frac{1}{2}{\displaystyle\sum_{x,y\in\overline{S}}}\left\vert u\left(x\right)-u\left(y\right)\right\vert^{p}\omega\left(x,y\right)+\displaystyle\sum_{z\in\Gamma}\frac{\sigma(z)}{\mu(z)}|u(z)|^{p}}{{\displaystyle\sum_{x\in S}}\left|u\left(x\right)\right|^{p}}
	\end{aligned}
	\]
	where $\mathcal{A}:=\left\{ u\,:\,\overline{S}\rightarrow\mathbb{R}\,|\, u\not\equiv 0\,\,\text{in}\,\, S,\,\,u=0\,\,\text{on}\,\,\partial S\setminus \Gamma\right\}$.
\end{Lemma}
In the above, the number $\lambda_{p,0}$ is called the first eigenvalue of $\Delta_{p,\omega}$ on a network $\overline{S}$ with corresponding eigenfunction $\phi_{0}$ (see \cite{C1} and \cite{CDS} for the spectral theory of the discrete Laplace operators). Here, we note that if $\Gamma$ is empty set, then $\sum_{z\in\Gamma}\frac{\sigma(z)}{\mu(z)}|u(z)|^{p}$ implies $0$.
\begin{Remark}
	It is clear that the first eigenvalue $\lambda_{p,0}$ is nonnegative. Moreover, we note here that the first eigenvalue $\lambda_{p,0}$ satisfies the following statements:
	\begin{itemize}
		\item[(i)] If $\sigma\equiv 0$, then $\lambda_{p,0}=0$.
		\item[(ii)] If $\sigma\not\equiv 0$, then $\lambda_{p,0}>0$.
	\end{itemize}
\end{Remark}

Now we will prove the existence of the solution to the equation \eqref{equation}, using the Schauder fixed point theorem. For this reason, we need the modified version of the Arzel\'a-Ascoli theorem as follows.

\begin{Lemma}[Modified version of the Arzel\'a-Ascoli theorem]\label{Arzela}	
	Let K be a compact subset of $\mathbb{R}$ and $\overline{S}$ be a network. Consider a Banach space $ C\left( \overline{S} \times K \right)$ with the maximum norm $\lVert u \rVert_{\overline{S},K} := \max_{x\in \overline{S}} \max_{t\in K} \left\vert u\left(x,t\right)\right\vert $. Then a subset $A$ of $C\left(\overline{S} \times K\right)$ is relatively compact if A is uniformly bounded on $\overline{S} \times K$ and $A$ is equicontinuous on $K$ for each $x\in \overline{S}$.
\end{Lemma}
\begin{proof}
	The proof of this version is similar to the original one (see \cite{B}). Thus we only state the idea of the proof. Let $\epsilon>0$ be arbitrarily given. Since $K$ is compact on $\mathbb{R}$ and A is equicontinuous on $K$, there is a finite open cover $\left\{N_{1}\left(t_i, \delta_i\right)\right\}$ of $K$ such that
	\begin{center}
		$\left\vert f\left(x,t\right)-f\left(x,t_i\right) \right\vert <\frac{\epsilon}{4} \hspace{1mm} \text{for all} \hspace{1mm} t\in N_{1}\left(t_i,\delta_i\right),\hspace{1mm} i=1,\dots,n ,\hspace{1mm} x \in \overline{S},\hspace{1mm} f \in A$.
	\end{center}
	Define $E=\left\{f\left(x,t_i\right)\vert\, x \in \overline{S} ,\, i=1,\dots,n,\, f \in A\right\} $. Then $E$ is totally bounded, since A is uniformly bounded. Hence there is a sequence $\left\{\xi_{j} \right\}^{m}_{j=1}$ in $\mathbb{R}$ such that
	\begin{center}
		$E \subset \bigcup^{m}_{j=1} N_{2}\left(\xi_{j},\frac{\epsilon}{4}\right)$.
	\end{center}
	Now, set $F:=\left\{k\,:\,\overline{S}\times\{1,\dots,n\}\rightarrow\{1,\dots,m\} \,\vert\, \text{k is a function}\right\}$ and define
	\begin{center}
		$A_k:=\{f\in A\,\vert\, f\left(x,t_i\right) \in N_{2}\left(\xi_{k_{\left(x,i\right)}},\frac{\epsilon}{4}\right),\,\, x\in\overline{S},\,\, i=1,\dots,n\}$ for each $k\in F$.
	\end{center}
	Then we have to show $A \subset \bigcup_{k\in K} A_{k}$. Let $f\in A$ be fixed. For each $x\in\overline{S}$, $i=1,\cdots,n$, $f\left(x,t_i\right)\in E\subset \bigcup_{j=1}^{m}N_{2}\left(\xi_{j},\frac{\epsilon}{4} \right)$. i.e. there is $j={k_{\left(x,i\right)}}$ such that $f\left(x,t_i \right)\in N_{2}\left(\xi_{j},\frac{\epsilon}{4}\right)$. Thus, $A \subset \bigcup_{k\in F} A_{k}$. 
	
	We now claim that the diameter of each $A_k$ is less than $\epsilon$. For each $f,g\in A_{k}$ and $(x,t)\in\overline{S}\times F$, there exists $1\leq i\leq n$ such that $t\in N_{1}\left(t_i,\delta_i\right)$ and
	\[
	\begin{aligned}
	\left\vert f\left(x,t\right)-g\left(x,t\right)\right\vert  \leq & \left\vert f\left(x,t\right)-f\left(x,t_i\right)\right\vert + \left\vert f\left(x,t_i\right)-\xi_{k_{\left(x,i\right)}} \right\vert \\ &  + \left\vert \xi_{k_{\left(x,i\right)}} - g\left(x,t_i\right) \right\vert + \left\vert g\left(x,t_i\right)-g\left(x,t\right)\right\vert < 4 \cdot \frac{\epsilon}{4}=\epsilon.
	\end{aligned}
	\]
	Hence, $A$ is totally bounded and the proof is complete.
\end{proof}
\begin{Theorem}[Local existence]\label{local existence}
	There exists $t_{0}>0$ such that the equation \eqref{equation} admits at least one solution $u$ such that $u(x,\cdot)$ is continuous on $[0,t_{0}]$ and differentiable in $(0,t_{0})$, for each $x\in \overline{S}$.
\end{Theorem}
\begin{proof}
	We first start with the following Banach space:
	\begin{equation*}
	C(S\times [0,t_{0}]):=\left\{ u:S\times [0,t_{0}]\rightarrow \mathbb{R} \,\vert \,u(x,\cdot)\in C([0,t_{0}])\,\,\text{for each}\,\,x\in S \right\}
	\end{equation*}
	with the maximum norm $\lVert u\rVert_{S,t_{0}}:=\max_{x\in S}\max_{0\leq t\leq t_{0}}|u(x,t)|$, where $t_{0}\in \mathbb{R}$ is a positive constant which will be defined later. Now, consider a subspace
	\begin{equation*}
	B_{t_{0}}:=\left\{ u\in C(S\times [0,t_{0}]) \,|\,\lVert u\rVert_{S,t_{0}}\leq 2\lVert u_{0}\rVert_{S,t_{0}} \right\}
	\end{equation*}
	of a Banach space $C(S\times [0,t_{0}])$. Then it is clear that $B_{t_{0}}$ is convex. In order to apply the Schauder fixed point theorem, we have to show that $B_{t_{0}}$ is closed. Let $f_{n}$ be a sequence in $B_{t_{0}}$ which converges to $f$. Since the convergence is uniform, $f$ is continuous. Moreover, $\left|\lVert f_{n} \rVert_{S,t_{0}}- \lVert f \rVert_{S,t_{0}} \right| \leq \lVert f_{n}-f\rVert_{S,t_{0}}$ implies that $f\in B_{t_{0}}$. Hence, $B_{t_{0}}$ is closed.
	
	Now, define a function $\psi:\mathbb{R}\rightarrow \mathbb{R}$ by
	\begin{equation*}
	\psi(\gamma):=\sum_{x\in S}|\gamma-u(x,t)|^{p-2}\left[\gamma-u(x,t)\right]a(x)+b|\gamma|^{p-2}\gamma,
	\end{equation*}
	where $a(x)\geq 0$ for all $x\in S$, $b\geq 0$ with $a(x)+b>0$ for some $x\in S$. Then it is easy to see that $\psi$ is a continuous function which is strictly increasing and bijective on $\mathbb{R}$. Therefore, there exists $\rho\in\mathbb{R}$ uniquely such that $\psi(\rho)=0$. It means that for all $u\in B_{t_{0}}$ and $(z,t)\in \partial S\times[0,t_{0}]$, we can define the value of $u(z,t)$ uniquely according to the boundary condition $B[u]=0$. i.e. for every $u\in B_{t_{0}}$, $u(z,t)$ satisfies
	$$
	\mu(z)\frac{\partial u}{\partial_{p} n}(z,t)+\sigma(z)|u(z,t)|^{p-2}u(z,t)=0,\,\,(z,t)\in \partial S\times[0,t_{0}]
	$$
	for all $(z,t)\in \partial S \times [0,t_{0}]$, where $\mu,\sigma:\partial S\rightarrow[0,+\infty)$ are given functions with $\mu(z)+\sigma(z)>0$ for all $z\in \partial S$.
	Then by the boundary condition, it is clear that $u(z,t)$ satisfies $|u(z,t)|\leq \lVert u\rVert_{S,t_{0}}$, $(z,t)\in \partial S\times [0,t_{0}]$.
	
	Let us define an operator $D:B_{t_{0}} \rightarrow B_{t_{0}}$ by
	\begin{equation*}
	D[u](x,t):=u_{0}(x)+\int_{0}^{t}\Delta_{p,\omega}u(x,s)+\lambda|u(x,s)|^{q-1}u(x,s)  \,ds,\,\,(x,t)\in S\times [0,t_{0}],
	\end{equation*}
	where $u_{0}:\overline{S}\rightarrow \mathbb{R}$ is a given function.

	Now, put $$t_{0}:=\frac{\max_{x\in S} u_{0}(x)}{\omega_{0}\left[4\max_{x\in S} u_{0}(x)\right]^{p-1}+\lambda \left[2\max_{x\in S} u_{0}(x)\right]^{q}},$$
	where $\omega_{0}:=\max_{x\in S}\sum_{y\in \overline{S}}d(x,y)$. Then it is easy to see that the operator $D$ is well-defined, in view of the definition of $t_{0}$.
	
	Now we will show that $D$ is continuous. The verification of the continuity is divided into 4 cases as follows: (i) $1<p<2$, $0<q<1$, (ii) $1<p<2$, $1\leq q$, (iii) $2\leq p$, $0<q<1$, and (iv) $2\leq p$, $2\leq q$. However, each case can be handled in a similar way with a little modification, here we handle the case (iii) only. For $u$ and $v$ in $B_{t_{0}}$, it follows that
	\begin{equation*}
	\begin{aligned}
	&\left\vert D[u](x,t)-D[v](x,t) \right\vert\\& \leq  \left\vert \int_{0}^{t} 2^{2p-3}(p-1)\lVert u_0 \rVert_{S,t_{0}}^{p-2}\lVert u-v \rVert_{S,t_{0}} \sum_{y\in \overline{S}}\omega(x,y) +2^{1-q}\lambda \lVert u-v \rVert_{S,t_{0}}^{q}  \right\vert
	\\& \leq  t_{0}\left[ 2^{2p-3}\left(p-1\right)\omega_{0}\lVert u_0 \rVert_{S,t_{0}}^{p-2}   \lVert u-v \rVert_{S,t_{0}}+2^{1-q}\lambda  \lVert u-v \rVert_{S,t_{0}}^{q} \right]
	\end{aligned}
	\end{equation*}
	for all $(x,t)\in S\times[0,t_{0}]$. Consequently, for each $p>1$ and $q>0$,
	$$ \lVert D[u]-D[v] \rVert_{S,t_{0}} \leq C_{1} \lVert u-v \rVert_{S,t_{0}}^{\min \{p-1,1\} } +C_{2} \lVert u-v \rVert_{S,t_{0}}^{\min\{q,1\}}+C_{3}\lVert u-v \rVert_{S,t_{0}} ,$$
	where $C_{1}$, $C_{2}$, and $C_{3}$ are constants depending only on $u_{0}$, $t_{0}$, $p$, $q$, $\omega_{0}$, and $\lambda$. Therefore we get the continuity of $D$.
	
	We will show that $D(B_{t_{0}})$ is uniformly bounded on $\overline{S}\times[0,t_{0}]$ and equicontinuous on $[0,t_{0}]$. Since $D$ is well-defined, $D(B_{t_{0}})\subset B_{t_{0}}$, it is clear that $D(B_{t_{0}})$ is uniformly bounded. On the other hand, it follows that for each $x\in\overline{S}$,
	\begin{equation*}
	\begin{aligned}
	\left\vert D[u]\left(x,t_1\right)-D[u]\left(x,t_2\right) \right\vert  \leq \left\vert t_1 - t_2 \right\vert \left[\omega_{0}\left(4\lVert u \rVert_{S,t_{0}}\right)^{p-1}+\lambda \left(2\lVert u \rVert_{S,t_{0}}\right)^{q}\right]
	\end{aligned}
	\end{equation*}
	for all $t_1,t_2\in[0,t_{0}]$, $u\in B_{t_{0}}$, which implies that $D(B_{t_{0}})$ is equicontinuous on $I$. Hence, $D(B_{t_{0}})$ is relatively compact by Theorem \ref{Arzela}, so that there is a function $u$ satisfying the equation \eqref{equation} on $S\times [0,t_{0}]$, by the Schauder fixed point theorem. Also, such $u$ satisfies the boundary condition $B[u]=0$. On the other hand, it is easy to see that $u$ is bounded and continuous on $\overline{S}\times [0,t_{0}]$. Moreover, $u(x,\cdot)$ is differentiable in $(0,t_{0})$, for each $x\in \overline{S}$ by the definition of $D$.
\end{proof}

Now, we discuss the comparison principles for the equation \eqref{equation}, in order to study the blow-up, extinctive occurrence, and global existence, which we begin in the next section.
\begin{Theorem}[Comparison Principle]\label{CP}
	Let $T>0$ ($T$ may be $+\infty$), $p>1$, $q\geq 1$, and $\lambda>0$. Suppose that real-valued functions $u(x,\cdot)$ and $v(x,\cdot)\in C[0,T)$ are differentiable in $(0,T)$ for each $x\in \overline{S}$ and satisfy
	\begin{equation}\label{cp equation}
	\begin{cases}
	u_{t}\left(x,t\right)-\Delta_{p,\omega}u\left(x,t\right)-\lambda|u(x,t)|^{q-1}u(x,t)  \\
	\geq v_{t}\left(x,t\right)-\Delta_{p,\omega}v\left(x,t\right)-\lambda|v(x,t)|^{q-1}v(x,t), & (x,t)\in S\times\left(0,T\right),\\
	\mu(x)\frac{\partial u}{\partial_{p} n}(x,t)+\sigma(x)\vert u(x,t)\vert^{p-2}u(x,t)\\
	\geq \mu(x)\frac{\partial u}{\partial_{p} n}(x,t)+\sigma(x)\vert u(x,t)\vert^{p-2}u(x,t), & (x,t)\in\partial S\times[0,T),\\
	u\left(x,0\right)\geq v\left(x,0\right), & x\in\overline{S}.
	\end{cases}
	\end{equation}
	Then $u(x,t)\geq v(x,t)$ for all $(x,t)\in \overline{S}\times [0,T)$.
\end{Theorem}
\begin{proof}
	Let $T'>0$ be arbitrarily given with $T'<T$. Then by mean value theorem, for each $x\in S$ and $0\leq t \leq T'$,
	\begin{equation*}
		|u(x,t)|^{q-1}u(x,t)-|v(x,t)|^{q-1}v(x,t)=q|\xi(x,t)|^{q-1}[u(x,t)-v(x,t)]
	\end{equation*}
	for some $\xi(x,t)$ lying between $u(x,t)$ and $v(x,t)$. Now, let us define functions $\tilde{u}, \tilde{v}\,:\,\overline{S}\times\left[0,T'\right]\rightarrow\mathbb{R}$ by
	\[
	\tilde{u}(x,t):=e^{-2\lambda qLt}u\left(x,t\right),\,\,
	(x,t)\in\overline{S}\times[0,T'],
	\]
	\[
	\tilde{v}(x,t):=e^{-2\lambda qLt}v\left(x,t\right),\,\,
	(x,t)\in\overline{S}\times[0,T'],
	\]
	where $L:=[\,\max_{x\in\overline{S}}\max_{0\leq t\leq T'}\{|u(x,t)|,|v(x,t)|\}\, ]^{q-1}$. Then from \eqref{cp equation}, we have
	\begin{equation}\label{eq1-4}
	\begin{aligned}
	&\left[\tilde{u}_{t}\left(x,t\right)-\tilde{v}_{t}\left(x,t\right)\right]-e^{2\lambda qL(p-2)t}\left[\Delta_{p,\omega}\tilde{u}\left(x,t\right)-\Delta_{p,\omega}\tilde{v}\left(x,t\right)\right]\\ &+\lambda q\left[2L-|\xi(x,t)|^{q-1}\right]\left[\tilde{u}\left(x,t\right)-\tilde{v}\left(x,t\right)\right] \geq 0
	\end{aligned}
	\end{equation}
	for all $\left(x,t\right)\in S\times(0,T']$.

	We recall that $\tilde{u}(x,\cdot)$ and $\tilde{v}(x,\cdot)$ are continuous on $[0,T']$ for each $x\in \overline{S}$ and $\overline{S}$ is finite. Hence, we can find
	$\left(x_{0},t_{0}\right)\in\overline{S}\times\left[0,T'\right]$ such that
	\[
	\left(\tilde{u}-\tilde{v}\right)\left(x_{0},t_{0}\right)={\displaystyle \min_{x\in\overline{S}}\min_{0\leq t\leq T'}\left(\tilde{u}-\tilde{v}\right)\left(x,t\right)},
	\]
	which implies that
	\begin{equation}\label{123}
	\tilde{v}\left(y,t_{0}\right)-\tilde{v}\left(x_{0},t_{0}\right)\leq \tilde{u}\left(y,t_{0}\right)-\tilde{u}\left(x_{0},t_{0}\right),\,\,\,y\in\overline{S}.
	\end{equation}
	Then now we have only to show that $\left(\tilde{u}-\tilde{v}\right)\left(x_{0},t_{0}\right)\geq0$.
	
	Suppose that $\left(\tilde{u}-\tilde{v}\right)\left(x_{0},t_{0}\right)<0$, on the contrary. Assume  that $x_{0}\in\partial S$. Then we see that
	\begin{equation}\label{boundary comparison principle}
	\begin{aligned}
	0 \leq&\mu\left(x_{0}\right)\sum_{x\in S}\left[|\tilde{u}\left(x_{0},t_{0}\right)-\tilde{u}\left(x,t_{0}\right)|^{p-2}\left(\tilde{u}\left(x_{0},t_{0}\right)-\tilde{u}\left(x,t_{0}\right)\right)\right.\\&-\left.|\tilde{v}\left(x_{0},t_{0}\right)-\tilde{v}\left(x,t_{0}\right)|^{p-2}\left(\tilde{v}\left(x_{0},t_{0}\right)-\tilde{v}\left(x,t_{0}\right)\right)\right]\omega\left(x_{0},x\right)\\
	&+\sigma\left(x_{0}\right)\left(\tilde{u}\left(x_{0},t_{0}\right)-\tilde{v}\left(x_{0},t_{0}\right)\right)
	\end{aligned}
	\end{equation}
	Therefore, if $\sigma(x_{0})>0$ then the equation \eqref{boundary comparison principle} is negative, which leads a contradiction. If $\sigma(x_{0})=0$, then we have
	\[
	\tilde{u}(x_{0},t_{0})-\tilde{v}(x_{0},t_{0})=\tilde{u}(x,t_{0})- \tilde{v}(x,t_{0})
	\]
	for all $x\in S$. Hence, there exists $x_{1}\in S$ such that 
	\[
	\tilde{u}(x_{0},t_{0})-\tilde{v}(x_{0},t_{0})=\tilde{u}(x_{1},t_{0})- \tilde{v}(x_{1},t_{0}).
	\]
	Hence, we can always choose $x_{0}\in S$. Since $\tilde{u}(x,0)-\tilde{v}(x,0)\geq0$ on $\overline{S}$, we have $\left(x_{0},t_{0}\right)\in S\times(0,T']$. Then we obtain from \eqref{123} that
	\begin{equation}\label{eq1-5}
	\Delta_{p,\omega}\tilde{u}\left(x_{0},t_{0}\right)-\Delta_{p,\omega}\tilde{v}\left(x_{0},t_{0}\right)\geq 0
	\end{equation}
	and it follows from the differentiability of $\left(\tilde{u}-\tilde{v}\right)\left(x,t\right)$ in $(0,T']$ for each $x\in\overline{S}$ that
	\begin{equation}\label{eq1-6}
	\left(\tilde{u}_{t}-\tilde{v}_{t}\right)\left(x_{0},t_{0}\right)\leq 0.
	\end{equation}
	Combining \eqref{eq1-4}, \eqref{eq1-5}, and \eqref{eq1-6}, we obtain
	\[
	\begin{aligned}
	0\leq&\left[\tilde{u}_{t}\left(x,t\right)-\tilde{v}_{t}\left(x,t\right)\right]-e^{2\lambda qL(p-2)t}\left[\Delta_{p,\omega}\tilde{u}\left(x,t\right)-\Delta_{p,\omega}\tilde{v}\left(x,t\right)\right]\\ &+\lambda q\left[2L-|\xi(x,t)|^{q-1}\right]\left[\tilde{u}\left(x,t\right)-\tilde{v}\left(x,t\right)\right]<0,
	\end{aligned}
	\]
	which leads a contradiction. Therefore, $u\left(x,t\right)\geq v\left(x,t\right)$ for all $(x,t)\in\overline{S}\times[0,T)$, since $T'<T$ is arbitrarily given.
\end{proof}
When $p\geq 2$, we obtain a strong comparison principle as follows:
\begin{Theorem}[Strong Comparison Principle]\label{SCP}
	Let $T>0$ ($T$ may be $+\infty$), $p\geq 2$, $q\geq 1$, and $\lambda>0$. Suppose that real-valued functions $u(x,\cdot)$ and $v(x,\cdot)\in C[0,T)$ are differentiable in $(0,T)$ for each $x\in \overline{S}$ and satisfy the inequality \eqref{cp equation}. If $u_{0}(x^{*})>v_{0}(x^{*})$ for some $x^{*}\in S$, then $u(x,t)>v(x,t)$ for all $(x,t)\in S\cup \Gamma\times (0,T)$.
\end{Theorem}
\begin{proof}
	First, note that $u\geq v$ on $\overline{S}\times[0,T)$ by theorem \ref{CP}. Let $T'>0$ be arbitrarily given with $T'<T$. Define functions $\tau\,:\,\overline{S}\times\left[0,T'\right]\rightarrow\mathbb{R}$ by
	\[
	\tau\left(x,t\right) := u\left(x,t\right)-v\left(x,t\right),\,\,\left(x,t\right)\in\overline{S}\times\left[0,T'\right].
	\]
	Then $\tau\left(x,t\right)\geq0$ for all $\left(x,t\right)\in\overline{S}\times\left[0,T'\right]$. From the inequality \eqref{cp equation}, we have
	\begin{equation}\label{098}
	\begin{aligned}
	\tau_{t}(x^{*},t)&\geq \Delta_{p,\omega}u(x^{*},t)-\Delta_{p,\omega}v(x^{*},t)
	\end{aligned}
	\end{equation}
	for all $0<t\leq T'$. Then by the mean value theorem, for each $y\in \overline{S}$ and $0\leq t\leq T'$, it follows that
	\begin{equation}\label{111}
	\begin{aligned}
	&|u(y,t)-u(x^{*},t)|^{p-2}[u(y,t)-u(x^{*},t)]-|v(y,t)-v(x^{*},t)|^{p-2}[v(y,t)-v(x^{*},t)]\\&=(p-1)|\zeta(x^{*},y,t)|^{p-2}[\tau(y,t)-\tau(x,t)],
	\end{aligned}
	\end{equation}
	where $|\zeta(x^{*},y,t)|\leq 2M$ and $M=\max_{x\in\overline{S}}\max_{0\leq t\leq T'}\{|u(x,t)|,|v(x,t)|\}$.	Using \eqref{111}, the inequality \eqref{098} becomes
	\begin{equation*}
	\begin{aligned}
	 \tau_{t}\left(x^{*},t\right) \geq-d_{\omega}x^{*}(p-1)[2M]^{p-2}\tau\left(x^{*},t\right).
	\end{aligned}
	\end{equation*}
	This implies
	\begin{equation}\label{equation1}
	\tau\left(x^{*},t\right)\geq\tau\left(x^{*},0\right)e^{-\left(d_{\omega}x^{*}(p-1)[2M]^{p-2}\right)t}>0,\,\,t\in(0,T'],
	\end{equation}
	since $\tau\left(x^{*},0\right)>0$. Now, suppose there exists $(x_{0},t_{0})\in S\cup\Gamma\times (0,T']$ such that
	\begin{center}
		$\tau\left(x_{0},t_{0}\right)=\displaystyle \min_{x\in S\cup\Gamma,\,0<t\leq T'} \tau\left(x,t\right)=0$.
	\end{center}
	Case 1: $x_{0}\in S$.\\
	Since $\tau(x_{0},t_{0})\leq \tau(x,t)$ for all $(x,t)\in \overline{S}\times [0,T']$, We have
	\[
	\tau_{t}\left(x_{0},t_{0}\right)\leq0
	\]
	and
	\[
	\Delta_{p,\omega}u\left(x_{0},t_{0}\right)-\Delta_{p,\omega}v\left(x_{0},t_{0}\right)\geq0.
	\]
	Hence, from the inequality \eqref{098}, we obtain
	\begin{center}
		$0\leq\tau_{t}\left(x_{0},t_{0}\right)-\Delta_{p,\omega}u\left(x_{0},t_{0}\right)+\Delta_{p,\omega}v\left(x_{0},t_{0}\right)\leq0$.
	\end{center}
	Therefore, we have
	\[
	\Delta_{p,\omega}u\left(x_{0},t_{0}\right)-\Delta_{p,\omega}v\left(x_{0},t_{0}\right)=0,
	\]
	which implies that $\tau\left(y,t_{0}\right)=0$ for all $y\in \overline{S}$ with $y\sim x_{0}$.
	Now, for any $x \in \overline S,$ there exists a path
	\begin{displaymath}
	x_{0} \sim x_{1} \sim \cdots \sim x_{n} \sim x,
	\end{displaymath}
	since $\overline S$ is connected. By applying the same argument as above inductively we see that $\tau(x, t_{0})=0$ for every $x \in \overline S$, which is a contradiction to \eqref{equation1}.\\
	Case 2: $x_{0}\in \Gamma$.\\
	By the boundary condition in \eqref{cp equation}, we have
	\begin{equation*}
	\begin{aligned}
		&\mu(x)\left[\frac{\partial u}{\partial_{p} n}(x_{0},t_{0})-\frac{\partial u}{\partial_{p} n}(x_{0},t_{0})\right]\\&\geq\sigma(x)[\vert u(x_{0},t_{0})\vert^{p-2}u(x_{0},t_{0})-\vert u(x,t)\vert^{p-2}u(x_{0},t_{0})]=0,
	\end{aligned}
	\end{equation*}
	which follows that
	\begin{equation*}
		\sum_{x\in S}\left[ -|u(x,t_{0})|^{p-2}u(x,t_{0})+|v(x,t_{0})|^{p-2}v(x,t_{0}) \right]\omega(x,x_{0})\geq 0.
	\end{equation*}
	It means that there exists $x_{1}\in S$ with $x_{0}\sim x_{1}$ such that $\tau(x_{1},t_{0})=0$, which contradicts to Case 1. Hence, we finally obtain that $u\left(x,t\right)>v\left(x,t\right)$ for all $\left(x,t\right)\in S\times\left(0,T\right)$, since $T'<T$ is arbitrarily given.
\end{proof}
The rest of this section is devoted to investigate the following lemma which is basic result induced by the boundary condition $B[u]=0$.
\begin{Lemma}\label{boundarylemma}
	The solution $u$ to the equation \eqref{equation} satisfies that for all $z^{*}\in \partial S$ and $t\geq 0$, there exists $x^{*}\in S$ with $x^{*}\sim z^{*}$ such that $u(x^{*},t)\geq u(z^{*},t)$.
\end{Lemma}
\begin{proof}
	For all $z^{*}\in \partial S$ and $t\geq 0$, we have from the boundary condition $B[u]=0$ that
	\begin{equation*}
		\sum_{x\in S}|u(x,t)-u(z^{*},t)|^{p-2}[u(x,t)-u(z^{*},t)]\omega(x,y)=\frac{\sigma(z)}{\mu(z)}|u(z^{*},t)|^{p-2}u(z^{*},t)\geq 0,
	\end{equation*}
	by Theorem \ref{CP}. Hence, it is easy to see that there exists $x^{*}\in S$ with $x^{*}\sim z^{*}$ such that $u(x^{*},t)\geq u(z^{*},t)$, which completes the proof.
\end{proof}

\section{Main results and proofs}
	In this section, we will characterize the parameters $p$ and $q$ completely to see when the solution blows up or exist globally. Moreover, we consider extinctive solution in the global existence. From now on, by a solution to the equation \eqref{equation} we mean that it is a solution given in Theorem \ref{local existence} with a maximal interval of existence $[0,T)$.
	\begin{definition}[Blow-up]
		We say that a solution $u$ to the equation \eqref{equation} blows up in finite time $T>0$, if there exists $x\in S$ such that $\left\vert u\left(x,t\right)\right\vert \rightarrow +\infty$ as $t\nearrow T^{-}$, or equivalently, $\sum_{x\in S}|u(x,t)|\rightarrow +\infty$ as $t\nearrow T^{-}$.
	\end{definition}
	Before getting into the main results, we recall the following elementary inequalities.
	\begin{equation}\label{mainineq}
	\begin{aligned}
	&\left[\sum_{i=1}^{n} t_{i}\right]^{p} \leq n^{p-1}\sum_{i=1}^{n} t_{i}^{p}\leq n^{p-1}\left[\sum_{i=1}^{n}t_{i}\right]^{p}, \hspace{3mm}\text{when} \hspace{1mm} p\geq 1, \\& \left[\sum_{i=1}^{n}t_{i}\right]^{p}\leq \sum_{i=1}^{n} t_{i}^{p}\leq n^{1-p}\left[\sum_{i=1}^{n}t_{i}\right]^{p} , \hspace{3mm}\text{when} \hspace{1mm} 0<p<1,
	\end{aligned}
	\end{equation}
	where $t_{i}\geq 0$ for all $i=1,\cdots, n$.

	As seen in the Figure in the introduction, the solutions to the equation \eqref{equation} may blow up or exist globally, or vanish, depending on the parameters $\mu$, $\sigma$, $p$ and $q$. In particular, if $\sigma \equiv 0$ (the case of the Neumann boundary condition), then we obtain the following result.
	\begin{Theorem}\label{neumann}
		Assume that $\sigma \equiv 0$. Then the solution $u$ satisfies
		\begin{equation*}
		\sum_{x\in S}u_{t}(x,t)=\lambda\sum_{x\in S}|u(x,t)|^{q-1}u(x,t).
		\end{equation*}
		It means that the solution $u$ blows up in finite time $T$ if and only if $q>1$ for every $\lambda>0$ and nontrivial initial data $u_{0}$.
	\end{Theorem}
	\begin{proof}
		Summing up over $S$ to the equation \eqref{equation}, we have
		\begin{equation}\label{neu1}
		\begin{aligned}
			\sum_{x\in S}u_{t}(x,t)=&\sum_{x\in S}\Delta_{p,\omega}u(x,t)+\sum_{x\in S}|u(x,t)|^{q-1}u(x,t)\\
			=&\sum_{x\in \overline{S}}\Delta_{p,\omega}u(x,t)-\sum_{z\in\partial S}\Delta_{p,\omega}u(x,t)+\lambda\sum_{x\in S}|u(x,t)|^{q-1}u(x,t)\\=&\lambda\sum_{x\in S}|u(x,t)|^{q-1}u(x,t).
		\end{aligned}
		\end{equation}
		Therefore, applying the inequality \eqref{mainineq} to \eqref{neu1} and solving the differential inequality, we obtain
		\begin{equation}\label{neu2}
			\sum_{x\in S}u(x,t)\geq \lambda|S|^{1-q}\left[\sum_{x\in S}u(x,t)\right]^{q}
		\end{equation}
		for $q>1$ and
		\begin{equation}\label{neu3}
			\sum_{x\in S}u(x,t)\leq \lambda|S|^{1-q}\left[\sum_{x\in S}u(x,t)\right]^{q}
		\end{equation}
		for $q< 1$. Solving the differential inequality \eqref{neu2} and \eqref{neu3}, we obtain
		\begin{equation*}
			\sum_{x\in S}u(x,t)\geq \left[\frac{1}{\left[\sum_{x\in S}u_{0}(x)\right]^{1-q}-\lambda|S|^{1-q}(q-1)t }\right]^\frac{1}{q-1}
		\end{equation*}
		for $q>1$ and
		\begin{equation*}
		\sum_{x\in S}u(x,t)\leq \left[\lambda|S|^{1-q}(1-q)t+\left[\sum_{x\in S}u_{0}(x)\right]^{1-q} \right]^\frac{1}{1-q}
		\end{equation*}
		for $q< 1$. Moreover, we can easily obtain that
		\begin{equation*}
			\sum_{x\in S}u(x,t)=e^{\lambda t}\sum_{x\in S}u_{0}(x)
		\end{equation*}
		for $q=1$. Hence, the solution $u$ blows up in finite time $T$ if and only if $q>1$ for every $\lambda>0$ and nontrivial initial data $u_{0}$.
	\end{proof}
	\begin{Remark}
		Assume that $\sigma\equiv 0$ and $q>1$. Then the blow-up time $T$ can be estimated as
		\begin{equation*}
			0<T\leq \frac{\left[ \sum_{x\in S}u_{0}(x) \right]^{1-q}}{\lambda (q-1)|S|^{1-q}}.
		\end{equation*}
	\end{Remark}
	\begin{Remark}
		The proof in the Theorem \ref{neumann} also tells us a behavior of the growth of the solutions. More preciesly, if $q<1$, the the solutions $u$ may increase polynomially in $t$. If $q=1$, then the solution $u$ increase exponentially in $t$.
	\end{Remark}
	From now on, we discuss the main results with the assumption $\sigma\not\equiv 0$. We now start with the case $0<p-1<q$ and $q>1$.
	\begin{Theorem}\label{blowup}
		Assume that $0<p-1<q$, $q>1$, and $\sigma \not\equiv 0$. Then the solution $u$ to the equation \eqref{equation} blows up in finite time $T>0$, provided that $\max_{x\in S}u_{0}(x)>\left(\frac{\max_{x\in S}d_{\omega}x}{\lambda}\right)^{\frac{1}{q-p+1}}$.
	\end{Theorem}
	\begin{proof}
		First, we note that the solution $u$ is nonnegative and exists uniquely by Theorem \ref{CP}. For each $t>0$, we can take $x^{*}\in S$ such that $\max_{x\in \overline{S}}u(x,t)=u(x^{*},t)$ by Lemma \ref{boundarylemma}. In fact, it is easy to see that $\max_{x\in \overline{S}}u(x,t)$ is differentiable for almost all $t>0$. Now, the equation \eqref{equation} can be written as
		\begin{equation}\label{blowup1}
		\begin{aligned}
			u_{t}(x^{*},t)=&\Delta_{p,\omega}u(x^{*},t)+\lambda u^{q}(x^{*},t) \\ =&\sum_{y\in \overline{S}}|u(y,t)-u(x^{*},t)|^{p-2}\left[u(y,t)-u(x^{*},t)\right]\omega(x,y)+\lambda u^{q}(x^{*},t) \\ \geq&-d_{\omega}x^{*}u^{p-1}(x^{*},t)+\lambda u^{q}(x^{*},t)\\ =& \lambda u^{q}(x^{*},t)\left[ 1-\frac{d_{\omega}x^{*}}{\lambda}u^{p-1-q}(x^{*},t) \right]
		\end{aligned}
		\end{equation}
		for almost all $t>0$. Therefore, if the initial data $u_{0}$ is so large in a sense that
		\begin{equation*}
			\max_{x\in S}u_{0}(x)>\left(\frac{\max_{x\in S}d_{\omega}x}{\lambda}\right)^{\frac{1}{q-p+1}},
		\end{equation*}
		then we obtain from \eqref{blowup1} that
		\begin{equation}\label{blowup2}
			u_{t}(x^{*},t)\geq C_{0}\lambda u^{q}(x^{*},t)
		\end{equation}
		for almost all $t>0$, where $C_{0}:=1-\frac{\max_{x\in S}d_{\omega}x}{\lambda}\max_{x\in S}u_{0}^{p-1-q}(x)$. Solving the differential inequality \eqref{blowup2}, we obtain
		\begin{equation*}
			u(x^{*},t)\geq \left[ \frac{1}{\left[\max_{x\in S}u_{0}(x)\right]^{1-q}-C_{0}\lambda(q-1)t } \right]^{\frac{1}{q-1}},
		\end{equation*}
		which implies that the solution $u$ blows up in finite time $0<T\leq\frac{\left[\max_{x\in S}u_{0}(x)\right]^{1-q}}{C_{0}\lambda(q-1)}$
	\end{proof}
	\begin{Remark}
		When the solution blows up in the above, the blow-up time $T$ can be estimated as
		\begin{equation*}
			0<T\leq \frac{\left[\max_{x\in S}u_{0}(x)\right]^{1-q}}{\left[1-\frac{\max_{x\in S}d_{\omega}x}{\lambda}\max_{x\in S}u_{0}^{p-1-q}(x)\right]\lambda(q-1)}.
		\end{equation*}
	\end{Remark}
	We now discuss the blow-up rate when the solution $u$ blows up in finite time $T$.
	\begin{Theorem}\label{rate}
		Assume that $0<p-1<q$ and $q>1$. Suppose the solution $u$ to the equation \eqref{equation} blows up in finite time $T$. Then the following statements are true:
		\begin{itemize}
			\item[(i)] $\max_{x\in S}u(x,t) \geq \left[ \lambda(q-1))(T-t) \right]^{-\frac{1}{q-1}}$, $0<t<T$.
			\item[(ii)] $\max_{x\in S}u(x,t) \leq \left[ \lambda(q-1)(T-t)-\alpha(T-t)^{\frac{2q-p}{q-1}} \right]^{-\frac{1}{q-1}}$, $0<t<T$.
			\item[(iii)] $\lim_{t\rightarrow T-} (T-t)^{\frac{1}{q-1}} \max_{x\in S}u(x,t)=\left[ \frac{1}{\lambda(q-1)} \right]^\frac{1}{q-1}$, $0<t<T$.
		\end{itemize}
	Here, $\alpha:=\lambda^{\frac{q-p+1}{q-1}}\max_{x\in S}d_{\omega}x(q-1)^{\frac{2q-p}{q-1}}$.
	\end{Theorem}
	\begin{proof}
		$\left(i\right)$. Firstly, we note that the solution $u$ to the equation \eqref{equation} is positive on $S\cup\Gamma \times(0,T)$, by Theorem \ref{SCP}. As in the previous theorem, let $x^{*}\in S$ be a node such that $u\left(x^{*},t\right):={\max_{x\in S}u\left(x,t\right)}$ for each $t>0$. Then it follows from the equation \eqref{equation} that
		\begin{eqnarray*}
			u_{t}\left(x^{*},t\right) & \leq & \lambda u^{q}\left(x^{*},t\right),
		\end{eqnarray*}
		for almost all $t>0$. Then integrating from $t$ to $T$, we get
		\begin{eqnarray*}
			\lambda(T-t) & \geq & \int_{t}^{T}\frac{u_{t}\left(x^{*},s\right)}{u^{q}\left(x^{*},s\right)}ds\\
			& = & \int_{u\left(x^{*},t\right)}^{+\infty}\frac{ds}{s^{q}}\\
			& = & \frac{u^{1-q}(x^{*},t)}{q-1}.
		\end{eqnarray*}
		Hence, we obtain
		\[
		u(x^{*},t)\geq \left[\lambda(q-1)(T-t)\right]^{-\frac{1}{q-1}},\,\,0<t<T.
		\]
		$(ii)$. Since the solution $u$ is positive, we get
		\begin{eqnarray*}
			u_{t}\left(x^{*},t\right)  & \geq & -\sum_{y\in\overline{S}}u^{p-1}\left(x^{*},t\right)\omega\left(x^{*},y\right)+\lambda u^{q}\left(x^{*},t\right)\\
			& \geq & -d_{\omega}x^{*} u^{p-1}\left(x^{*},t\right)+\lambda u^{q}\left(x^{*},t\right)\\
			& = & u^{q}\left(x^{*},t\right)\left[\lambda-\max_{x\in S}d_{\omega}x u^{p-1-q}\left(x^{*},t\right)\right]\\
		\end{eqnarray*}
		for almost all $t>0$. Then it follows from $(i)$ that
		\[
		u_{t}(x^{*},t) \geq  u^{q}\left(x^{*},t\right)\left[\lambda-\max_{x\in S}d_{\omega}x \left[\lambda(q-1)(T-t)\right]^{\frac{q-p+1}{q-1}}\right].\\
		\]
		Integrating from $t$ to $T$, we get
		\[
		u(x^{*},t)\leq \left[\lambda(q-1)(T-t)-\alpha(T-t)^{\frac{2q-p}{q-1}}\right]^{-\frac{1}{q-1}}, \,\,0<t<T.
		\]
		where $\alpha:=\lambda^{\frac{q-p+1}{q-1}}\max_{x\in S}d_{\omega}x(q-1)^{\frac{2q-p}{q-1}}$.\\
		Finally, $(iii)$ can be easily obtained by $(i)$ and $(ii)$.
	\end{proof}
	\begin{Remark}
		In Theorem \ref{rate}, we can easily see that the blow-up rate does not depend on the boundary condition $B[u]=0$.
	\end{Remark}
	Now, we discuss the case $0<q<p-1$.
	\begin{Theorem}\label{bound}
		Assume that $0<q<p-1$ and $\sigma\not\equiv 0$. Then every solution $u$ to the equation \eqref{equation} is global. More precisely, every solution $u$ satisfies
		\begin{equation*}
		\begin{aligned}
		&\sum_{x\in S}u^{2}(x,t)\\&\leq \max\left\{\sum_{x\in S}u_{0}(x),\left[\frac{\lambda|S|^\frac{p-2}{2}}{\lambda_{p,0}}\right]^\frac{2}{p-q-1},\left[\frac{\lambda|S|^\frac{p-q-1}{2}}{\lambda_{p,0}}\right]^\frac{2}{p-q-1}, \left[\frac{\lambda|S|^\frac{1-q}{2}}{\lambda_{p,0}}\right]^\frac{2}{p-q-1}\right\}
		\end{aligned}
		\end{equation*}
		for all $t\geq 0$.
	\end{Theorem}
	\begin{proof}
		Multiplying \eqref{equation} by $u$ and summing up over $\overline{S}$, we obtain from the boundary condition $B[u]=0$, Lemma \ref{laplace theorem}, and Lemma \ref{eigenvalue} that
		\begin{equation}\label{global1}
		\begin{aligned}
		&\frac{1}{2}\frac{d}{dt}\sum_{x\in S}u^{2}\left(x,t\right)\\=& \sum_{x\in \overline{S}}\left[\Delta_{p,\omega}u(x,t)\right] u(x,t)-\sum_{z\in\partial S}\left[\Delta_{p,\omega}u(x,t)\right] u(x,t)+\lambda\sum_{x\in S}|u(x,t)|^{q+1}\\=&-\frac{1}{2}\sum_{x\in\overline{S}} \left\vert u\left(y,t\right)-u\left(x,t\right)\right\vert^{p}\omega\left(x,y\right)-\sum_{z\in\Gamma}\frac{\sigma(z)}{\mu(z)}|u(z)|^{p} +\lambda \sum_{x\in S}\left\vert u\left(x,t\right)\right\vert^{q+1}\\\leq &-\lambda_{p,0}\sum_{x\in S}|u(x,t)|^{p}+\lambda\sum_{x\in S}|u(x,t)|^{q+1}.
		\end{aligned}
		\end{equation}
		Now, we divide this proof into 3 cases.
		\newline Case 1 : $q\geq 1$ and $p>2$.\\
		Applying the inequality \eqref{mainineq} to \eqref{global1}, we obtain
		\begin{equation*}
			\frac{1}{2}\frac{d}{dt}\sum_{x\in S}u^{2}(x,t)\leq \left[\sum_{x\in S}u^{2}(x,t)\right]^{\frac{q+1}{2}}\left[ \lambda-\lambda_{p,0}|S|^\frac{2-p}{2}\left[\sum_{x\in S}u^{2}(x,t)\right]^\frac{p-q-1}{2} \right].
		\end{equation*}
		Therefore, if there exists $t_{1}\in (0,\infty)$ such that
		$$
		\lambda<\lambda_{p,0}|S|^\frac{2-p}{2}\left[\sum_{x\in S}u^{2}(x,t_{1})\right]^\frac{p-q-1}{2},
		$$
		then we have $\frac{1}{2}\frac{d}{dt}\sum_{x\in S}u^{2}(x,t_{1})<0$. Also, if there exists $t_{2}\in (0,\infty)$ such that 
		$$
		\lambda\geq\lambda_{p,0}|S|^\frac{2-p}{2}\left[\sum_{x\in S}u^{2}(x,t_{2})\right]^\frac{p-q-1}{2},
		$$
		then it is easy to see that the solution $u$ must satisfy
		\begin{equation*}
			\lambda\geq\lambda_{p,0}|S|^\frac{2-p}{2}\left[\sum_{x\in S}u^{2}(x,t)\right]^\frac{p-q-1}{2}
		\end{equation*}
		for all $t\in[t_{2},\infty)$. That is to say, the solution $u$ satisfies
		\begin{equation}\label{global111}
		\sum_{x\in S}u^{2}(x,t)\leq \max\left\{\sum_{x\in S}u_{0}(x),\left[\frac{\lambda|S|^\frac{p-2}{2}}{\lambda_{p,0}}\right]^\frac{2}{p-q-1}\right\}
		\end{equation}
		for all $t\geq 0$.
		\newline Case 2 : $0<q<1$ and $p\geq 2$.\\
		Applying the inequality \eqref{mainineq} to \eqref{global1}, it follows that
		\begin{equation}
		\begin{aligned}
			&\frac{1}{2}\frac{d}{dt}\sum_{x\in S}u^{2}(x,t)\\&\leq \lambda|S|^\frac{1-q}{2} \left[\sum_{x\in S}u^{2}(x,t)\right]^\frac{q+1}{2}\left[ 1-\frac{\lambda_{p,0}}{\lambda|S|^\frac{p-q-1}{2}}\left[\sum_{x\in S}u^{2}(x,t)\right]^\frac{p-q-1}{2} \right].
		\end{aligned}
		\end{equation}
		Hence, by the same argument as Case 1, we have
		\begin{equation}\label{global112}
			\sum_{x\in S}u^{2}(x,t)\leq \max\left\{\sum_{x\in S}u_{0}(x),\left[\frac{\lambda|S|^\frac{p-q-1}{2}}{\lambda_{p,0}}\right]^\frac{2}{p-q-1}\right\}
		\end{equation}
		for all $t\geq 0$.\\
		\newline Case 3 : $0<q<1$ and $1<p<2$.\\
		Applying the inequality \eqref{mainineq} to \eqref{global1}, we have
		\begin{equation}
		\begin{aligned}
		&\frac{1}{2}\frac{d}{dt}\sum_{x\in S}u^{2}(x,t)\\&\leq \lambda|S|^\frac{1-q}{2} \left[\sum_{x\in S}u^{2}(x,t)\right]^\frac{q+1}{2}\left[ 1-\frac{\lambda_{p,0}}{\lambda|S|^\frac{1-q}{2}}\left[\sum_{x\in S}u^{2}(x,t)\right]^\frac{p-q-1}{2} \right].
		\end{aligned}
		\end{equation}
		Therefore, by the same argument as Case 1, we obtain
		\begin{equation}\label{global113}
		\sum_{x\in S}u^{2}(x,t)\leq \max\left\{\sum_{x\in S}u_{0}(x),\left[\frac{\lambda|S|^\frac{1-q}{2}}{\lambda_{p,0}}\right]^\frac{2}{p-q-1}\right\}
		\end{equation}
		for all $t\geq 0$.\\
		Combining \eqref{global111}, \eqref{global112}, and \eqref{global113}, we finally obtain that
		\begin{equation*}
		\begin{aligned}
			&\sum_{x\in S}u^{2}(x,t)\\&\leq \max\left\{\sum_{x\in S}u_{0}(x),\left[\frac{\lambda|S|^\frac{p-2}{2}}{\lambda_{p,0}}\right]^\frac{2}{p-q-1},\left[\frac{\lambda|S|^\frac{p-q-1}{2}}{\lambda_{p,0}}\right]^\frac{2}{p-q-1}, \left[\frac{\lambda|S|^\frac{1-q}{2}}{\lambda_{p,0}}\right]^\frac{2}{p-q-1}\right\}
		\end{aligned}
		\end{equation*}
		for all $t\geq 0$.
	\end{proof}
	Now we discuss the case $0<p-1\leq q$ with $q\leq 1$.
	\begin{Theorem}\label{global}
		Assume that $0<p-1\leq q$, $q\leq 1$, and $\sigma\not\equiv 0$. Then every solution $u$ to the equation \eqref{equation} is global.
	\end{Theorem}
	\begin{proof}
		Multiplying \eqref{equation} by $u$ and summing up over $\overline{S}$, we obtain from the boundary condition $B[u]=0$ and Lemma \ref{laplace theorem} that
		\begin{equation}\label{glo1}
		\begin{aligned}
		&\frac{1}{2}\frac{d}{dt}\sum_{x\in S}u^{2}\left(x,t\right)\\=& \sum_{x\in \overline{S}}\left[\Delta_{p,\omega}u(x,t)\right] u(x,t)-\sum_{z\in\partial S}\left[\Delta_{p,\omega}u(x,t)\right] u(x,t)+\lambda\sum_{x\in \overline{S}}|u(x,t)|^{q+1}\\=&-\frac{1}{2}\sum_{x\in\overline{S}} \left\vert u\left(y,t\right)-u\left(x,t\right)\right\vert^{p}\omega\left(x,y\right)-\sum_{z\in\Gamma}\frac{\sigma(z)}{\mu(z)}|u(z)|^{p} +\lambda \sum_{x\in S}\left\vert u\left(x,t\right)\right\vert^{q+1}\\\leq & \lambda\sum_{x\in S}|u(x,t)|^{q+1}.
		\end{aligned}
		\end{equation}
		Therefore, applying the inequality \eqref{mainineq} to \eqref{glo1}, we obtain
		\begin{equation*}
			\frac{d}{dt}\sum_{x\in S}u^{2}(x,t)\leq 2 \lambda|S|^\frac{1-q}{2}\left[\sum_{x\in S}u^{2}(x,t)\right]^{\frac{q+1}{2}},
		\end{equation*}
		which implies that
		\begin{equation*}
			\sum_{x\in S}u^{2}(x,t)\leq \left[ \left[\sum_{x\in S}u_{0}^{2}(x)\right]^\frac{1-q}{2}+\lambda|S|^\frac{1-q}{2}(1-q)t \right]^\frac{2}{1-q}
		\end{equation*}
		for $q<1$ and
		\begin{equation*}
			\sum_{x\in S}u^{2}(x,t)\leq \left[\sum_{x\in S}u_{0}^{2}(x)\right] e^{2\lambda t}
		\end{equation*}
		for $q=1$.
	\end{proof}
	\begin{Remark}
		The proof in the Theorem \ref{global} also tells us a behavior of the growth of the solutions. More preciesly, if $q<1$, the the solutions $u$ may increase polynomially in $t$. If $q=1$, then the solution $u$ may increase exponentially in $t$.
	\end{Remark}
	Now we discuss the case $0<p-1<q$ with $1<p<2$ to investigate the extinctive solutions.
	\begin{Theorem}\label{extinction}
		Assume that $0<p-1<q$, $1<p<2$, and $\sigma\not\equiv 0$. Then every solution $u$ to the equation \eqref{equation} vanishes in finite time $T$, provided that the initial data $u_{0}$ is so small that $\lambda |S|^\frac{1}{2} \left[\sum_{x\in S}u_{0}^{2}(x)  \right]^{\frac{q-p+1}{2}} < \lambda_{p,0}.$
	\end{Theorem}
	\begin{proof}
		Multiplying \eqref{equation} by $u$ and summing up over $\overline{S}$, we obtain from the boundary condition $B[u]=0$, Lemma \ref{laplace theorem}, and Lemma \ref{eigenvalue} that
		\begin{equation}\label{extinc1}
		\begin{aligned}
		&\frac{1}{2}\frac{d}{dt}\sum_{x\in S}u^{2}\left(x,t\right)\\=& \sum_{x\in \overline{S}}\left[\Delta_{p,\omega}u(x,t)\right] u(x,t)-\sum_{z\in\partial S}\left[\Delta_{p,\omega}u(x,t)\right] u(x,t)+\lambda\sum_{x\in \overline{S}}|u(x,t)|^{q+1}\\=&-\frac{1}{2}\sum_{x\in\overline{S}} \left\vert u\left(y,t\right)-u\left(x,t\right)\right\vert^{p}\omega\left(x,y\right)-\sum_{z\in\Gamma}\frac{\sigma(z)}{\mu(z)}|u(z)|^{p} +\lambda \sum_{x\in S}\left\vert u\left(x,t\right)\right\vert^{q+1}\\\leq &-\lambda_{p,0}\sum_{x\in S}|u(x,t)|^{p}+\lambda\sum_{x\in S}|u(x,t)|^{q+1}.
		\end{aligned}
		\end{equation}
		Now, we divide this proof into 2 cases.
		\newline Case 1 : $0<q<1$.
		\newline Applying the inequality \eqref{mainineq} to \eqref{extinc1}, we obtain
		\begin{equation*}
		\begin{aligned}
			\frac{d}{dt}\sum_{x\in S}u^{2}(x,t)&\leq -2\lambda_{p,0}\left[\sum_{x\in S}u^{2}(x,t)\right]^{\frac{p}{2}}+2\lambda|S|^\frac{1-q}{2}\left[\sum_{x\in S}u^{2}(x,t)\right]^\frac{q+1}{2}\\&\leq -2\left[ \sum_{x\in S}u^{2}(x,t) \right]^\frac{p}{2}\left[ \lambda_{p,0}-\lambda|S|^\frac{1-q}{2}\left[\sum_{x\in S}u^{2}(x,t)\right]^\frac{q-p+1}{2} \right].
		\end{aligned}
		\end{equation*}
		Therefore, if the initial data is so small that
		\begin{equation*}
			\lambda_{p,0}-\lambda|S|^\frac{1-q}{2}\left[\sum_{x\in S}u_{0}^{2}(x)\right]^\frac{q-p+1}{2}>0,
		\end{equation*}
		then we arrive at
		\begin{equation}\label{extinc2}
			\frac{d}{dt}\sum_{x\in S}u^{2}(x,t)\leq -C_{3}\left[ \sum_{x\in S}u^{2}(x,t) \right]^\frac{p}{2}
		\end{equation}
		for all $t>0$, where $C_{3}:=2\lambda_{p,0}-2\lambda|S|^\frac{1-q}{2}\left[\sum_{x\in S}u_{0}^{2}(x)\right]^\frac{q-p+1}{2}$. Hence, solving the differential inequality \eqref{extinc2}, we obtain
		\begin{equation*}
			\sum_{x\in S}u^{2}(x,t)\leq \left[ \left[\sum_{x\in S}u_{0}^{2}(x)\right]^\frac{2-p}{2} -\frac{(2-p)C_{3}t}{2} \right]_{+}^\frac{2}{2-p},
		\end{equation*}
		which implies that the solution $u$ vanishes in finite time $0<T\leq\frac{2\left[\sum_{x\in S}u_{0}^{2}\left(x\right)\right]^\frac{2-p}{2}}{(2-p)C_{3}}$.
		\newline Case2 2 : $q\geq 1$.
		\newline Applying the inequality \eqref{mainineq} to \eqref{extinc1}, we obtain
		\begin{equation*}
		\begin{aligned}
		\frac{d}{dt}\sum_{x\in S}u^{2}(x,t)&\leq -2\lambda_{p,0}\left[\sum_{x\in S}u^{2}(x,t)\right]^{\frac{p}{2}}+2\lambda\left[\sum_{x\in S}u^{2}(x,t)\right]^\frac{q+1}{2}\\&\leq -2\left[ \sum_{x\in S}u^{2}(x,t) \right]^\frac{p}{2}\left[ \lambda_{p,0}-\lambda\left[\sum_{x\in S}u^{2}(x,t)\right]^\frac{q-p+1}{2} \right].
		\end{aligned}
		\end{equation*}
		Therefore, if the initial data is so small that
		\begin{equation*}
		\lambda_{p,0}-\lambda\left[\sum_{x\in S}u_{0}^{2}(x)\right]^\frac{q-p+1}{2}>0,
		\end{equation*}
		then we arrive at
		\begin{equation}\label{extinc3}
		\frac{d}{dt}\sum_{x\in S}u^{2}(x,t)\leq -C_{4}\left[ \sum_{x\in S}u^{2}(x,t) \right]^\frac{p}{2}
		\end{equation}
		for all $t>0$, where $C_{4}:=2\lambda_{p,0}-2\lambda\left[\sum_{x\in S}u_{0}^{2}(x)\right]^\frac{q-p+1}{2}$. Hence, solving the differential inequality \eqref{extinc2}, we obtain
		\begin{equation*}
		\sum_{x\in S}u^{2}(x,t)\leq \left[ \left[\sum_{x\in S}u_{0}^{2}(x)\right]^\frac{2-p}{2} -\frac{(2-p)C_{4}t}{2} \right]_{+}^\frac{2}{2-p},
		\end{equation*}
		which implies that the solution $u$ vanishes in finite time $0<T\leq\frac{2\left[\sum_{x\in S}u_{0}^{2}\left(x\right)\right]^\frac{2-p}{2}}{(2-p)C_{4}}$.
	\end{proof}
	\begin{Remark}
		When the solutions extinct in the above, the extinction time $T$ can be estimated as
		\begin{equation*}
		0<T\leq \frac{\left[\sum_{x\in S}u_{0}^{2}\left(x\right)\right]^\frac{2-p}{2}}{(2-p)\left[\lambda_{p,0}-\lambda|S|^\frac{1}{2}\left[\sum_{x\in S}u_{0}^{2}(x)\right]^\frac{q-p+1}{2} \right]}.
		\end{equation*}
	\end{Remark}
	Now, we will discuss the critical case $p-1=q>0$. Firstly, we investigate the case $p-1=q>1$.
	\begin{Theorem}\label{criticalblow}
		Assume that $p-1=q>1$ and $\sigma\not\equiv 0$. Then the solution $u$ to the equation \eqref{equation} satisfies the following statements.
		\begin{itemize}
			\item[(i)] If $\lambda >\lambda_{p,0}$, then the solution $u$ blows up in finite time $T$ for every $\lambda>0$ and nontrivial initial data $u_{0}$.
			\item[(ii)] If $\lambda = \lambda_{p,0}$, then solution $u$ exists globally. Moreover, the solution $u$ has an upper bound. 
			\item[(iii)] If $\lambda < \lambda_{p,0}$, then solution $u$ exists globally. Moreover, the solution $u$ may decrease polynomially in $t$.
		\end{itemize}
	\end{Theorem}
	\begin{proof}
		First of all, we note that the solution $u$ to the equation \eqref{equation} is nonnegative and exists uniquely. Also, the first eigenvalue $\lambda_{p,0}>0$, since $\sigma\not\equiv 0$. In this proof, we denote $M:=\max_{x\in S}\phi_{0}(x)$ and $m=\min_{x\in S}\phi_{0})(x)$.\\
		(i). Take $t_{0}>0$ to be so small that the solution $u$ doesn't blow up before $t_{0}$. In fact, existence of such $t_{0}$ can be guaranteed by Theorem \ref{local existence}. Now consider the following ODE problem:
		\begin{equation*}
		\begin{cases}
			g'(t)=m^{p-2}(\lambda-\lambda_{p,0})g^{p-1}(t),\,\,t>t_{0},\\
			g(t_{0})=\frac{\min_{x\in S\cup\Gamma}u(x,t_{0})}{M}.
		\end{cases}
		\end{equation*}
		Here, we can easily obtain that $g(t_{0})>0$ by Theorem \ref{SCP}. Solving the above ODE problem, we have
		\begin{equation*}
			g(t)=\left[ \frac{1}{g^{2-p}(t_{0})-(p-2)m^{p-2}(\lambda-\lambda_{p,0})(t-t_{0})} \right]^\frac{1}{p-2},\,\,t\geq t_{0},
		\end{equation*}
		which implies that $g$ blows up in finite time $T^{*}$. we now define $v(x,t):=g(t)\phi_{0}(x)$ for all $(x,t)\in \overline{S}\times[t_{0},T^{*})$. Then we see that $v(x,t_{0})=g(t_{0})\phi_{0}(x)\leq u(x,t_{0})$ for all $x\in \overline{S}$ and
		\begin{equation*}
		\begin{aligned}
			&\mu(z)\frac{\partial v}{\partial_{p}n}(z,t)+\sigma(z)|v(z,t)|^{p-2}v(z,t)\\&=g^{p-1}(t)\left[\mu(z)\frac{\partial \phi_{0}}{\partial_{p}n}(z)+\sigma(z)|\phi_{0}(z)|^{p-2}\phi_{0}(z)\right]=0
		\end{aligned}
		\end{equation*}
		for all $(z,t)\in \partial S\times [t_{0},T^{*})$. Moreover, we have
		\begin{equation*}
		\begin{aligned}
			&v_{t}(x,t)-\Delta_{p,\omega}v(x,t)-\lambda v^{p-1}(x,t)\\&=\phi_{0}(x)g'(t)+\lambda_{p,0}\phi_{0}^{p-1}(x)g^{p-1}(t)-\lambda \phi_{0}^{p-1}(x)g^{p-1}(t)\\&\leq(m^{p-2}\phi_{0}(x)-\phi_{0}^{p-1}(x))(\lambda-\lambda_{p,0})g^{p-1}(t)\leq 0
		\end{aligned}
		\end{equation*}
		for all $(x,t)\in S\times[t_{0},T^{*})$. Hence, $v(x,t)\leq u(x,t)$ for all $(x,t)\in \overline{S}\times[t_{0},T^{*})$ by Theorem \ref{CP}, which implies that $u$ blows up in finite time $t_{0}<T\leq T^{*}$.\\
		(ii). Take $v(x,t)=k\phi_{0}(x)$ for all $(x,t)\in \overline{S}\times[0,\infty)$, where $k:=\frac{\max_{x\in S\cup\Gamma}u_{0}(x)}{m}$. Then we have $u_{0}(x)\leq v(x,0)$ for all $x\in \overline{S}$,
		\begin{equation*}
			\mu(z)\frac{\partial v}{\partial_{p}}(z,t)+\sigma(z)|v(z,t)|^{p-2}v(z,t)=0
		\end{equation*}
		for all $(z,t)\in \partial S\times[0,\infty)$, and
		\begin{equation*}
			v_{t}(x,t)-\Delta_{p,\omega}v(x,t)-\lambda v^{p-1}(x,t)=\lambda_{p,0}k^{p-1}\phi_{0}^{p-1}(x)-\lambda k^{p-1}\phi_{0}^{p-1}(x)=0,
		\end{equation*}
		which implies that $u(x,t)\leq v(x,t)\leq kM$ for all $(x,t)\in \overline{S}\times[0,\infty)$.\\
		(iii). Consider the following ODE problem:
		\begin{equation*}
		\begin{cases}
		g'(t)=M^{p-2}(\lambda-\lambda_{p,0})g^{p-1}(t),\,\,t>0,\\
		g(0)=\frac{\max_{x\in\overline{S}}u_{0}(x)}{m}.
		\end{cases}
		\end{equation*}
		Solving the above ODE problem, we have
		\begin{equation}\label{ggg}
		g(t)=\left[ \frac{1}{g^{2-p}(0)+(p-2)M^{p-2}(\lambda_{p,0}-\lambda)(t-t_{0})} \right]^\frac{1}{p-2},\,\,t\geq 0,
		\end{equation}
		which implies that $g$ exists globally. Take $v(x,t)=g(t)\phi_{0}(x)$ for all $(x,t)\in \overline{S}\times[0,\infty)$. Then we see that $u_{0}(x)\leq v(x,0)=g(t_{0})\phi_{0}(x)$ for all $x\in \overline{S}$ and
		\begin{equation*}
		\begin{aligned}
		&\mu(z)\frac{\partial v}{\partial_{p}}(z,t)+\sigma(z)|v(z,t)|^{p-2}v(z,t)\\&=g^{p-1}(t)\left[\mu(z)\frac{\partial \phi_{0}}{\partial_{p}}(z)+\sigma(z)|\phi_{0}(z)|^{p-2}\phi_{0}(z)\right]=0
		\end{aligned}
		\end{equation*}
		for all $(z,t)\in \partial S\times [t_{0},\infty)$. Moreover, we have
		\begin{equation*}
		\begin{aligned}
		&v_{t}(x,t)-\Delta_{p,\omega}v(x,t)-\lambda v^{p-1}(x,t)\\&=\phi_{0}(x)g'(t)+\lambda_{p,0}\phi_{0}^{p-1}(x)g^{p-1}(t)-\lambda \phi_{0}^{p-1}(x)g^{p-1}(t)\\&\geq(M^{p-2}\phi_{0}(x)-\phi_{0}^{p-1}(x))(\lambda-\lambda_{p,0})g^{p-1}(t)\geq 0
		\end{aligned}
		\end{equation*}
		for all $(x,t)\in S\times[t_{0},T^{*})$. Hence, $u(x,t)\leq v(x,t)$ for all $(x,t)\in \overline{S}\times[0,\infty)$ by Theorem \ref{CP}, which means that $u$ exists globally. Moreover, \eqref{ggg} gives us that the solution $u$ may decrease polynomially in $t$.
	\end{proof}
	\begin{Remark}
		Assume that $p-1=q>1$, $\sigma\not \equiv 0$, and $\lambda>\lambda_{p,0}$. Then the solution to the equation \eqref{equation} blows up in finite time $T$ for every $\lambda>0$ and nontrivial initial data $u_{0}$. In this case, we estimate the blow-up time $T$ roughly as follows:
		\begin{equation*}
		\begin{aligned}
			0<T\leq& t_{0}+ \frac{\left[\max_{x\in S}\phi_{0}(x)\right]^{p-2}}{(p-2)(\lambda-\lambda_{p,0})\left[\min_{x\in S\cup\Gamma}u(x,t_{0})\right]^{p-2}\min_{x\in S}\left[\phi_{0}(x)\right]^{p-2}}\\\leq &t_{0}+ \frac{\left[\max_{x\in S}\phi_{0}(x)\right]^{p-2}}{(p-2)(\lambda-\lambda_{p,0})\left[2\max_{x\in S}u_{0}(x)\right]^{p-2}\min_{x\in S}\left[\phi_{0}(x)\right]^{p-2}}.
		\end{aligned}
		\end{equation*}
		Here,
		$$t_{0}:=\frac{\max_{x\in S} u_{0}(x)}{\omega_{0}\left[4\max_{x\in S} u_{0}(x)\right]^{p-1}+\lambda \left[2\max_{x\in S} u_{0}(x)\right]^{q}}$$
		which led from the Theorem \ref{local existence}.
	\end{Remark}
	Now we discuss the critical case $0<p-1=q\leq 1$ and $\sigma\not\equiv 0$. Actually, we already have the result that every solution $u$ to the equation \eqref{equation} exists globally by Theorem \ref{global}. Hence, our purpose is to know when the solution $u$ vanishes in finite time $T$.
	 \begin{Theorem}\label{criticalextinc}
		Assume that $0<p-1=q<1$ and $\sigma\not\equiv 0$. If $\lambda <\lambda_{p,0}$, then every solution $u$ to the equation \eqref{equation} vanished in finite time $T>0$.
	\end{Theorem}
	\begin{proof}
		Multiplying \eqref{equation} by $u$ and summing up over $\overline{S}$, we obtain from Lemma \ref{eigenvalue} that
		\begin{equation}\label{criex1}
		\begin{aligned}
			&\frac{1}{2}\frac{d}{dt}\sum_{x\in S}u^{2}\left(x,t\right)\\=& \sum_{x\in \overline{S}}\left[\Delta_{p,\omega}u(x,t)\right] u(x,t)-\sum_{z\in\partial S}\left[\Delta_{p,\omega}u(x,t)\right] u(x,t)+\lambda\sum_{x\in \overline{S}}|u(x,t)|^{p}\\=&-\frac{1}{2}\sum_{x\in\overline{S}} \left\vert u\left(y,t\right)-u\left(x,t\right)\right\vert^{p}\omega\left(x,y\right)-\sum_{z\in\Gamma}\frac{\sigma(z)}{\mu(z)}|u(z)|^{p} +\lambda \sum_{x\in S}\left\vert u\left(x,t\right)\right\vert^{p}\\\leq &-(\lambda_{p,0}-\lambda)\sum_{x\in S}|u(x,t)|^{p}.
		\end{aligned}
		\end{equation}
		Therefore, applying the inequality \eqref{mainineq} to \eqref{criex1}, we obtain
		\begin{equation*}
			\frac{1}{2}\frac{d}{dt}\sum_{x\in S}u^{2}(x,t)\leq -(\lambda_{p,0}-\lambda)\left[ \sum_{x\in S}u^{2}(x,t) \right]^\frac{p}{2},
		\end{equation*}
		which implies that
		\begin{equation*}
			\sum_{x\in S}u^{2}(x,t)\leq \left[ \left[\sum_{x\in S}u_{0}(x)\right]^\frac{2-p}{2}-(2-p)(\lambda_{p,0}-\lambda)t \right]_{+}^\frac{2}{2-p}.
		\end{equation*}
		Here, $(a)_{+}:=\max \{0,a\}$. Hence, there exists $T>0$ such that $\sum_{x\in S}u^2(x,T)=0$, which means that every solution $u$ vanishes in finite time $T$.
		\begin{Remark}
			When the solutions extinct in the above, the extinction time $T$ can be estimated as
			\begin{equation*}
			0<T\leq \frac{\left[\sum_{x\in S}u_{0}\right]^\frac{2-p}{2}}{(2-p)(\lambda_{p,0}-\lambda)}.
			\end{equation*}
		\end{Remark}
	\end{proof}
\section{Numerical illustration}
In this section, we exploit our result in the previous section with numerical experiments. Through this section, we consider a graph $S=\{x_{1}, x_{2},x_{3}, x_{4}\}$ with the boundary $\partial S=\{x_{5}, x_{6}\}$ and the weight given by Figure \ref{fig1}.
\begin{figure}[ht]
	\scriptsize
	\begin{displaymath}
	\xymatrix{
		& & \bullet_{x_{2}}\ar@{-}_{2}[r]\ar@{-}_{2}[ld]&\bullet_{x_{4}}\ar@{..}_{2}[r]\ar@{-}_{1}[ld]&\circ_{x_{6}}\\
		\circ_{x_{5}}\ar@{..}_{1}[r]&\bullet_{x_{1}}\ar@{-}_{2}[r]&\bullet_{x_{3}}\\
	}
	\end{displaymath}
	\caption{\label{fig1}Graph $\overline{S}$}
\end{figure}
Here, the number on (or under) each edge denotes the weight.\\

Firstly we consider the case $\sigma \equiv 0$ (Neumann boundary condition).\\
By the boundary condition $B[u]=0$, we obtain
$$
u(x_{5},t)=u(x_{1},t),\,\,\text{and}\,\,u(x_{6},t)=u(x_{4},t)
$$
for all $t\geq 0$.
\begin{example}[$\sigma\equiv 0$]
	Consider the case $p=3$, $q=2$, $\lambda=2$, and $\sigma \equiv 0$ (Neumann boundary condition). Put the initial data $u_{0}$ by $u_{0}(x_{1})=u_{0}(x_{2})=u_{0}(x_{3})=u_{0}(x_{5})=0$ and $u_{0}(x_{4})=u_{0}(x_{6})=0.1$. Figure $2$ show the solution to the equation \eqref{equation} which exploit the Theorem \ref{neumann}.
	\begin{figure}[ht]
		\centering
		\includegraphics[width=6cm,height=5cm]{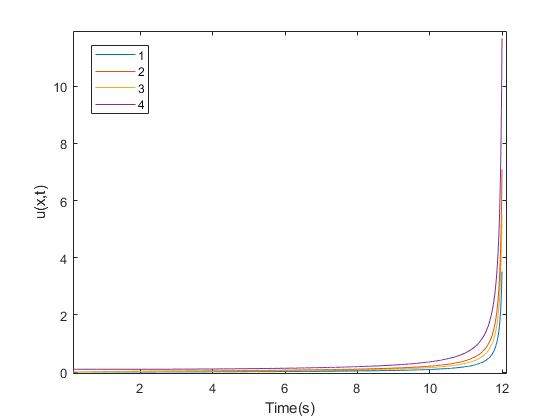}
		\caption{Blow-up solution for $q>1$ under the Neumann boundary condition}
	\end{figure}
	We can see that the solution blows up in finite time even though the initial data is small.
\end{example}
\begin{example}[$\sigma\not\equiv 0$]
	Consider the case $p=3$, $q=0.5$, $\lambda=2$, and $\sigma \equiv 0$ (Neumann boundary condition). Put the initial data $u_{0}$ by $u_{0}(x_{1})=u_{0}(x_{2})=u_{0}(x_{3})=u_{0}(x_{5})=u_{0}(x_{4})=u_{0}(x_{6})=7$. Figure $3$ show the solution to the equation \eqref{equation} which exploit the Theorem \ref{neumann}.
	\begin{figure}[ht]
		\centering
		\includegraphics[width=6cm,height=5cm]{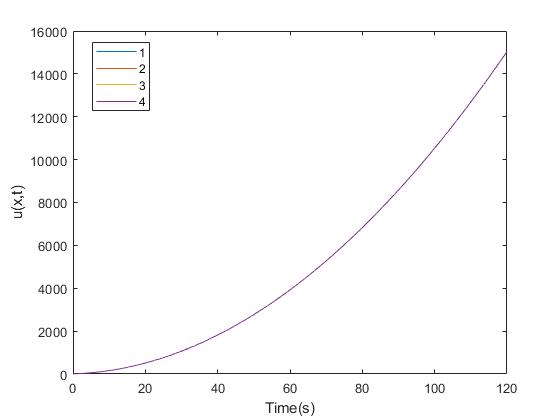}
		\caption{Global solution for $q\leq 1$ under the Neumann boundary condition}
	\end{figure}
\newpage
	We can see that the solution exists globally even though the initial data is large.
\end{example}
Next, we discuss the case $\sigma \not \equiv 0$. From now on, we only consider the case $\mu(x_{5})=\sigma(x_{5})=1$, $\mu(x_{6})=1$, and $\sigma(x_{6})=0$. Then we obtain
$$
u(x_{6},t)=\frac{1}{2}u(x_{4},t),\,\,\text{and}\,\,u(x_{6},t)=u(x_{4},t)
$$
for all $t\geq 0$. Also, we consider two types of initial data $u_{0}$ and $u_{1}$ as belows.
	\begin{table}[ht]\label{table}
	\scriptsize
	\centering
	\begin{tabular}{||c|c|c|c|c|c|c||}
		\hline
		& $x_{1}$ & $x_{2}$ & $x_{3}$ & $x_{4}$ & $x_{5}$ & $x_{6}$\\ \hline
		$u_{0}(x_{i})$ & 2& 1& 0& 1&1 &1 \\ \hline
		$u_{1}(x_{i})$ & 2& 1& 1& 2&1 &2 \\ \hline
	\end{tabular}
\end{table}\\
Then we can easily see that the initial data $u_{0}$ and $u_{1}$ satisfy the boundary condition $B[u]=0$.
\begin{example}[$0<p-1<q$]
	Consider the case $p=1.5$, $q=2$, and the initial data $u_{0}$. Then Figure $4$ show the blow-up solution and extinctive solution to the equation \eqref{equation} with the case $\lambda=3$ and $\lambda=0.1$, respectively.
	\begin{figure}[ht]
		\centering
		\includegraphics[width=6cm,height=5cm]{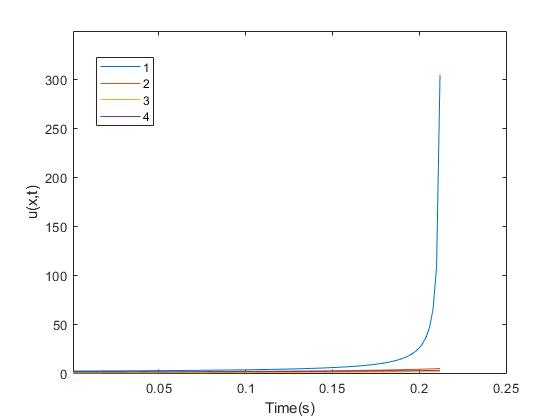}
		\includegraphics[width=6cm,height=5cm]{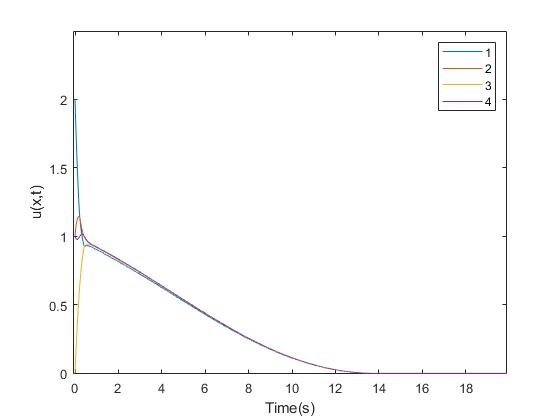}
		\caption{Blow-up solution for $\lambda=3$ (left) and extinctive solution for $\lambda=0.1$ (right)}
	\end{figure}
	\newpage
	In fact, we see that $\lambda_{p,0} \doteqdot 0.29 $ when $p=1.5$. By Theorem \ref{blowup} and Theorem \ref{extinction}, the solution $u$ blows up in finite time if $\lambda>\frac{\max_{x\in S}d_{\omega}x}{\max_{x\in S}u_{0}^{q-p+1}(x)}\doteqdot 1.77$, and the solution $u$ vanishes in finite time if $\lambda<\frac{\lambda_{p,0}}{\vert S\vert^{\frac{1-q}{2}}\left[ \sum_{x\in S}u_{0}^{2}(x) \right]^\frac{q-p+1}{2} }\doteqdot 0.15$.
\end{example}

\begin{example}[$0<p-1<q$]
	Consider the case $p=1.3$, $q=0.8$, and  $\lambda=0.18$. Then Figure $5$ illustrate the result of the Theorem \ref{extinction} with initial data $u_{0}$ to the extinctive solution and $u_{1}$ to the nonextinctive solution, respectively.
	\begin{figure}[ht]
		\centering
		\includegraphics[width=6cm,height=5cm]{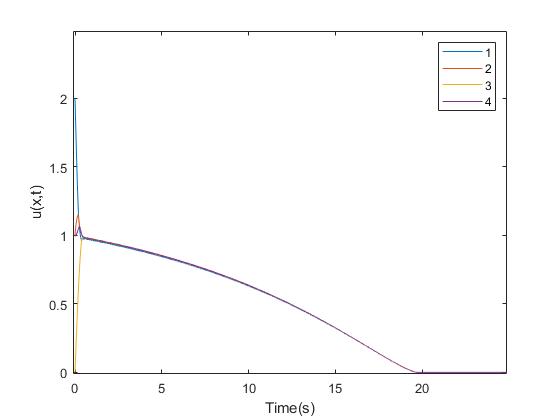}
		\includegraphics[width=6cm,height=5cm]{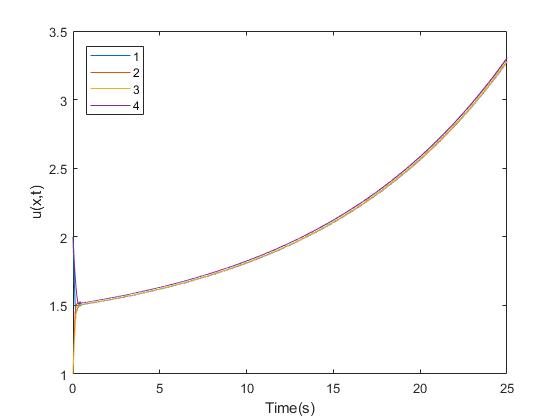}
		\caption{Extinctive solution for $u_{0}$ (left) and nonextinctive solution for $u_{1}$ (right)}
	\end{figure}
\end{example}
\begin{example}[$0<q<p-1$]
	Consider the case $\lambda=2$ and the initial data $u_{0}$. Then Figure $6$ show the bounded solution of the equation \eqref{equation} with the case  $p=3,\,q=0.5$ and $p=3,\,q=1.5$, which exploit the result of Theorem \ref{bound}.
	\begin{figure}[ht]
		\centering
		\includegraphics[width=6cm,height=5cm]{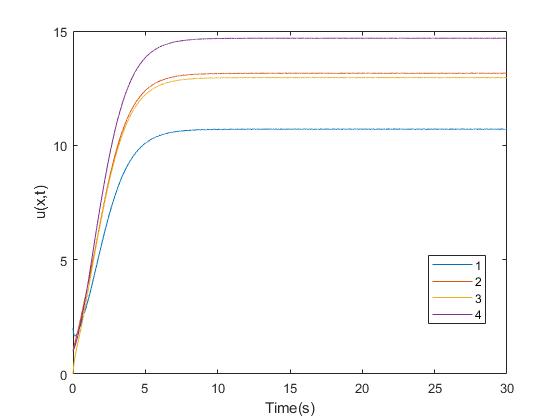}
		\includegraphics[width=6cm,height=5cm]{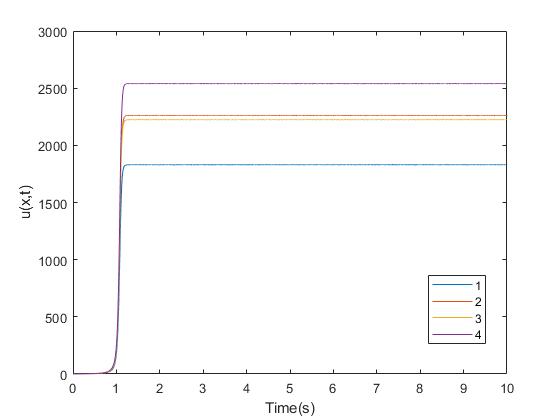}
		\caption{Bounded solution for $p=3,\,q=0.5$ (left) and $p=3,\,q=1.5$ (right)}
	\end{figure}
\end{example}

\newpage
\begin{example}[$p-1=q>1$]
	Consider the case $p-1=q=1.7$ and initial data $u_{1}$. Then Figure $7$ illustrate the result of the Theorem \ref{criticalblow} with the case $\lambda=0.5$ and $\lambda=0.05$, respectively.
	\begin{figure}[ht]
		\centering
		\includegraphics[width=6cm,height=5cm]{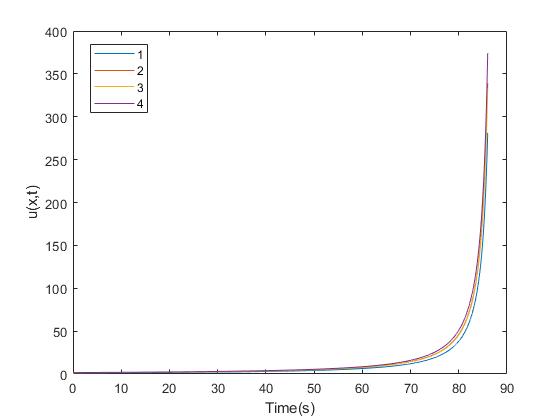}
		\includegraphics[width=6cm,height=5cm]{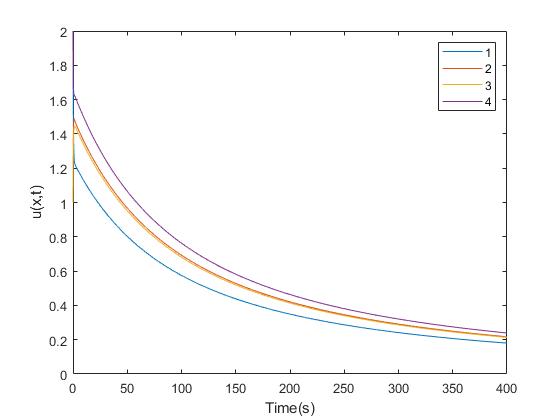}
		\caption{Blow-up solution for $\lambda=0.07$ (left) and global solution for $\lambda=0.05$ (right).}
	\end{figure}
	In fact, we see that $\lambda_{p,0} \doteqdot 0.06$ when $p=2.7$. By Theorem \ref{criticalextinc}, the solution $u$ blows up in finite time if $\lambda>\lambda_{p,0}$, and the solution $u$ exists globally if $\lambda\leq\lambda_{p,0}$.
\end{example}
\begin{example}[$p-1=q<1$]
	Consider the case $p-1=q=0.4$ and initial data $u_{1}$. Then Figure $8$ illustrate the result of the Theorem \ref{criticalextinc} with the case $\lambda=0.1$ and $\lambda=0.3$, respectively.
	\begin{figure}[ht]
		\centering
		\includegraphics[width=6cm,height=5cm]{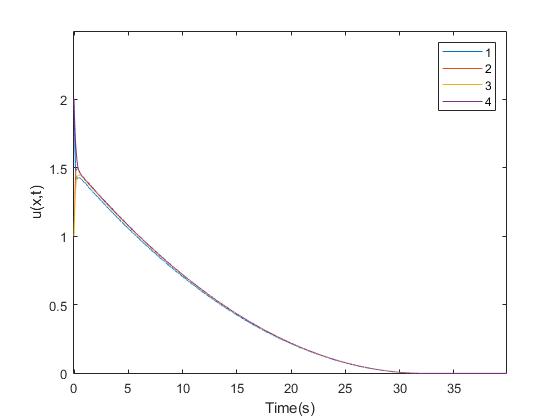}
		\includegraphics[width=6cm,height=5cm]{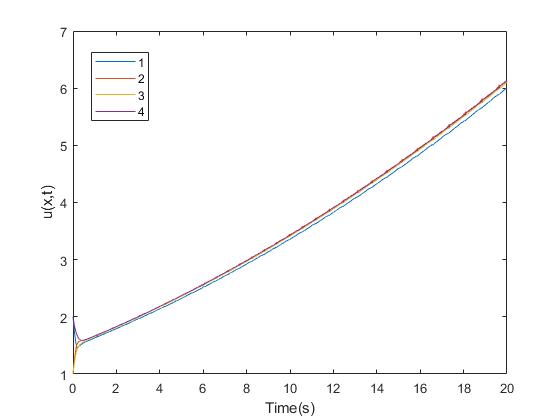}
		\caption{Extinctive solution for $\lambda=0.1$ (left) and global solution for $\lambda=0.3$ (right).}
	\end{figure}
	\newpage
	In fact, we see that $\lambda_{p,0} \doteqdot 0.204 $ when $p=1.5$. By Theorem \ref{criticalextinc}, the solution $u$ vanishes in finite time if $\lambda>\lambda_{p,0}$, and the solution $u$ exists globally if $\lambda\leq\lambda_{p,0}$.
\end{example}

\section*{Acknowledgments}
The first author is supported by Basic Science Research Program through the National Research Foundation of Korea (NRF) funded by the Ministry of Education (NRF-2015R1D1A1A01059561).

\section*{Conflict of Interests}
The authors declares that there is no conflict of interests regarding the publication of this paper.


\section*{References}

\begin{thebibliography}{20}
	
	\bibitem{YJ}
	J. Yin, C. Jin, \emph{Critical extinction and blow-up exponents for fast diffusive p-Laplacian with sources}, Math. Methods Appl. Sci., 30 (2007), no. 10, 1147–1167.
	
	\bibitem{LX}
	Y. Li, C.Xie, \emph{Blow-up for p-Laplacian parabolic equations}, Electron. J. Differential Equations, (2003), no. 20, 12 pp.
	
	\bibitem{C2}
	 Y.-G. Chen, \emph{Blow-up solutions of a semilinear parabolic equation with the Neumann and Robin boundary conditions}, J. Fac. Sci. Univ. Tokyo Sect. IA Math., 37 (1990), no. 3, 537–574.
	 
	\bibitem{L}
	X. Liu, \emph{Asymptotic behaviors of radially symmetric solutions to diffusion problems with Robin boundary condition in exterior domain}, Nonlinear Anal. Real World Appl., 39 (2018), 1–13. 
	 
	\bibitem{K}
	A. Kheloufi, \emph{On parabolic equations with mixed Dirichlet-Robin type boundary conditions in a non-rectangular domain}, Mediterr. J. Math., 13 (2016), no. 4, 1787–1805.
	
	\bibitem{GMPR}
	J. Garcia-Azorero, J. J. Manfredi, I. Peral, J. D. Rossi, \emph{Partial differential equations—the limit as p→∞ for the p-Laplacian with mixed boundary conditions and the mass transport problem through a given window}, Atti Accad. Naz. Lincei Rend. Lincei Mat. Appl., 20 (2009), no. 2, 111–126.
	
	\bibitem{Z}
	H, Zhang, \emph{Blow-up solutions and global solutions for nonlinear parabolic equations with mixed boundary conditions}, J. Appl. Math. Comput., 32 (2010), no. 2, 535–545. 
	
	\bibitem{B}
	R. F. Brown, \emph{A Topological Introduction to Nonlinear Analysis}, Birkh\"auser Boston, Inc., Boston, MA, $1993$.
	
	\bibitem{CB1}
	S.-Y. Chung and C. A. Berenstein, \emph{$\omega$-harmonic functions and inverse conductivity problems on networks}, SIAM J. Appl. Math., vol. $65$, no. $4$, $(2005)$, pp. $1200-1226$.
	
	\bibitem{CB2}
	J.-H. Kim and S.-Y. Chung, \emph{Comparison principles for the $p$-Laplacian on nonlinear networks}, J. Difference Equ. Appl., 16, (2010), no. 10, 1151-1163.
	
	\bibitem{KC}
	J.-H. Kim and S.-Y. Chung, \emph{Comparison principles for the $p$-Laplacian on nonlinear networks}, J. Difference Equ. Appl., 16, (2010), no. 10, 1151-1163.
	
	\bibitem{C1}
	F. R. Chung, \emph{Spectral Graph Theory}, CBMS regional Conference Series in Mathematics, American Mathematical Society, $1997$.
	
	\bibitem{CDS}
	D. M. Cvetkovi$\acute{c}$, M. Doob, and H. Sachs, \emph{Spectra of Graphs: Theory and Applications}, Academic Press, New York, NY, USA, $1980$.
	
	\bibitem{CH}
	S.-Y. Chung and J. Hwang, \emph{The discrete $p$-Schr\"{o}dinger equations under the mixed boundary conditions on networks}, preprint.
	
\end{thebibliography}
\end{document}